\documentclass[]{interact}
\usepackage{mathrsfs} 
\DeclareMathOperator{\diag}{diag}
\DeclareMathOperator{\diam}{diam}
\DeclareMathOperator*{\esup}{ess\,sup}
\DeclareMathOperator{\id}{id}
\DeclareMathOperator{\spann}{span}	
\DeclareMathOperator{\lip}{lip}		
\DeclareMathOperator{\sgn}{sgn}
\newcommand{\C}{{\mathbb C}}
\newcommand{\I}{{\mathbb I}}
\newcommand{\K}{{\mathbb K}}
\newcommand{\N}{{\mathbb N}}
\newcommand{\R}{{\mathbb R}}
\newcommand{\Z}{{\mathbb Z}}
\newcommand{\fA}{{\mathfrak A}}
\newcommand{\fT}{{\mathfrak T}}

\newcommand{\cO}{{\mathscr O}}
\newcommand{\cV}{{\mathscr V}}
\newcommand{\cW}{{\mathscr W}}
\newcommand{\sF}{{\mathscr F}}
\newcommand{\sG}{{\mathscr G}}
\newcommand{\sH}{{\mathscr H}}
\newcommand{\sK}{{\mathscr K}}
\newcommand{\sL}{{\mathscr L}}
\newcommand{\sN}{{\mathscr N}}

\renewcommand{\d}{\,{\mathrm d}}
\newcommand{\eps}{\varepsilon}
\newcommand{\tm}{\times}
\newcommand{\vphi}{\varphi}
\newcommand{\intoo}[1]{\left(#1\right)}				
\newcommand{\intcc}[1]{\left[#1\right]}				
\newcommand{\set}[1]{\left\{#1\right\}}				
\newcommand{\abs}[1]{\left|#1\right|}				
\newcommand{\norm} [1]{\left\|#1\right\|}			
\newcommand{\iprod} [1]{\left\langle#1\right\rangle}
\newcommand{\fall}{\quad\text{for all }}

\newcommand{\cref}[1]{Cor.~\ref{#1}}

\newcommand{\fref}[1]{Fig.~\ref{#1}}

\newcommand{\pref}[1]{Prop.~\ref{#1}}
\newcommand{\tref}[1]{Thm.~\ref{#1}}
\newcommand{\sref}[1]{Sect.~\ref{#1}}
\newtheorem*{hyp}{Hypothesis}
\theoremstyle{plain}
\newtheorem{theorem}{Theorem}[section]

\newtheorem{corollary}[theorem]{Corollary}
\newtheorem{proposition}[theorem]{Proposition}

\theoremstyle{definition}

\newtheorem{example}[theorem]{Example}

\theoremstyle{remark}
\newtheorem{remark}[theorem]{Remark}
\allowdisplaybreaks
\begin{document}

\articletype{ARTICLE TEMPLATE}

\title{Numerical dynamics of integrodifference equations:\\ 
Hierarchies of invariant bundles in $L^p(\Omega)$}

\author{
\name{Christian P\"otzsche\thanks{Email: christian.poetzsche@aau.at}}
\affil{Institut f\"ur Mathematik, Universit\"at Klagenfurt, 9020 Klagenfurt, Austria}
}
\maketitle
\begin{abstract}
	We study how the ''full hierarchy" of invariant manifolds for nonautonomous integrodifference equations on the Banach spaces of $p$-integrable functions behaves under spatial discretization of Galerkin type. These manifolds include the classical stable, center-stable, center, center-unstable and unstable ones, and can be represented as graphs of $C^m$-functions. For kernels with a smoothing property, our main result establishes closeness of these graphs in the $C^{m-1}$-topology under numerical discretizations preserving the convergence order of the method. It is formulated in a quantitative fashion and technically results from a general perturbation theorem on the existence of invariant bundles (i.e.\ nonautonomous invariant manifolds).
\end{abstract}
\begin{amscode}
65P99, 37J15, 37L25, 47H30, 45M10
\end{amscode}
\section{Introduction}
Over the last years, integrodifference equations (abbreviated IDEs) 
\begin{equation}
	\tag{$I_0$}
	u_{t+1}=\int_\Omega k_t(\cdot,y)g_t(y,u_t(y))\d y
\end{equation}
became popular and widely used models in theoretical ecology for the temporal evolution and spatial dispersal of populations having non-overlapping generations \cite{lutscher:19}. They constitute an interesting class of infinite-dimensional dynamical systems and can be seen as a discrete-time counterpart to reaction-diffusion equations, with whom they share various dynamical features. 
IDEs also arise in other fields, such as time-$1$-maps of evolutionary differential equations or as iterative schemes to solve nonlinear boundary value problems via an equivalent fixed point formulation (cf.~\cite[pp.~168--169]{martin:76}). 
Due to their origin in ecology, a periodic $t$-dependence in \eqref{ide0} is well-motivated, but recently also applications requiring more general temporal forcing became relevant \cite{jacobsen:jin:lewis:15}. 

The long term behavior of IDEs is often illustrated using numerical simulations, which require to discretize their state space $X$. For this reason it is a well-motivated question to relate the dynamics of the original and of the numerically approximated problem. This is a key issue in Numerical Dynamics \cite{stuart:humphries:98}. Among the various aspects of this field, we tackle invariant manifolds. These sets provide the skeleton of the state space for a dynamical system, since stable manifolds $W_s$ might serve as boundary between different domains of attraction, unstable manifolds $W_u$ constitute global attractors and center manifolds $W_{cs},W_c,W_{cu}$ capture the essential dynamics near non-hyperbolic solutions. These sets fulfill the \emph{classical hierarchy} (a notion from~\cite{aulbach:87})
\begin{equation}
	\begin{array}{ccccc}
		W_s & \subset & W_{cs} & \subset & X\\
		&&\cup&&\cup\\
		&&W_c & \subset & W_{cu}\\
		&&&&\cup\\
		&&&&W_u.
	\end{array}
	\label{chier}
\end{equation}
In continuous time, early contributions for autonomous ODEs in $\R^d$ were \cite{beyn:87} (stable and unstable manifolds) and \cite{beyn:lorenz:87} (center manifolds). Generalizations to infinite-di\-men\-sio\-nal problems appeared in \cite{alouges:debussche:91} (evolutionary PDEs), \cite{farkas:01} (delay equations) and \cite{jones:stuart:95} (full discretizations of evolutionary PDEs). The effect of time-discretizations to invariant manifolds for nonautonomous ODEs in Banach spaces was studied in \cite{keller:poetzsche:06}. 

Throughout the literature the state space of IDEs typically consists of continuous or integrable functions \cite{lutscher:19}. It is all the more important to understand the behavior of their invariant manifolds in the classical hierarchy \eqref{chier} when such equations are spatially discretized. Accordingly the paper at hand further explores the numerical dynamics of nonautonomous equations. For related work we refer to \cite{poetzsche:20a} on the classical situation of local stable and unstable manifolds for hyperbolic periodic solutions to Urysohn IDEs in the space of (H\"older) continuous functions over a compact habitat. The present approach complements and extends \cite{poetzsche:20a} in various aspects: First, we tackle Hammerstein IDEs \eqref{ide0} in $L^p$-spaces rather than over the continuous functions. Second, we address the complete hierarchy of invariant manifolds containing also strongly stable, the three types of center, as well as strongly unstable manifolds and establish convergence including their derivatives preserving the order of the discretization method. Third, dealing with general nonautonomous equations requires various alternative tools such as the dichotomy spectrum \cite{russ:16} (instead of the Floquet spectrum) and a flexible perturbation theorem for invariant bundles replacing the Lipschitz inverse function theorem used in \cite{poetzsche:20a}, whose applicability is restricted to the hyperbolic situation. We finally point out that such extended hierarchies of invariant manifolds date back to \cite{aulbach:87} in the framework of autonomous ODEs. 

The crucial assumption in our analysis are ambient smoothing properties of the kernels $k_t$ paving the way to corresponding error estimates. In order to minimize technicalities, we deal with globally defined semilinear IDEs having the trivial solution and therefore obtain global results. By means of well-known translation and cut-off methods more general and real-life problems under local assumptions can be adapted to the present setting and then yield results valid in the vicinity of a given reference solution (cf.\ \sref{seccutoff} or \cite[pp.~256ff, Sect.~4.6]{poetzsche:10}). 

This paper is structured as follows: In \sref{sec2} we provide the basic assumptions on nonautonomous Hammerstein IDEs such that they are well-defined on $L^p$-spaces and generate a completely continuous process. In order to circumvent the pathological smoothness properties of Nemytskii operators on $L^p$-spaces this requires them to satisfy suitable mapping properties into a larger $L^q$-space, $q<p$, being balanced by Hille-Tamarkin conditions on the kernel (cf.\ \sref{sec21}). In short, dealing with an IDE \eqref{ide0} of class $C^m$ requires exponents $p>m$. The starting point of our actual analysis are nonautonomous linear IDEs. Based on exponential dichotomies and the dichotomy spectrum we present the 'linear algebra' necessary for our analysis, where spectral intervals and bundles extend eigenvalue moduli respectively generalized eigenspaces (cf.~\cite{russ:16}) to a time-variant setting. The following \sref{sec3} contains our main results (Thms.~\ref{thmpifb} and \ref{thmcenter}) describing how nonautonomous invariant manifolds (bundles) and their derivatives up to order $m-1$ behave under discretization using projection methods. Concrete examples of projection methods and their applicability in Thms.~\ref{thmpifb} and \ref{thmcenter} are discussed in \sref{sec4}. Finally, for the reader's convenience, we conclude the paper with two appendices. App.~\ref{appA} contains the central perturbation result from \cite{poetzsche:03} in a formulation suitable for those not familiar with the calculus on measure chains. App.~\ref{appB} quotes a well-definedness criterion for linear Fredholm integral operators and (due to the lack of a suitable reference) provides well-definedness and smoothness properties for Nemytskii operators between Lebesgue spaces. 
\paragraph*{Notation}
Let $\K$ be one of the fields $\R$ or $\C$. A discrete interval $\I$ is the intersection of a real interval with the integers $\Z$ and $\I':=\set{t\in\I:\,t+1\in\I}$. Given $\ell\in\N$ and normed spaces $X,Y$, we write $L_\ell(X,Y)$ for the space of bounded $\ell$-linear operators $T:X^\ell\to Y$ and moreover $L_0(X,Y):=Y$, as well as $L(X,Y):=L_1(X,Y)$. It is handy to abbreviate $L_\ell(X):=L_\ell(X,X)$ and $L(X):=L(X,X)$. We write $N(S):=S^{-1}(0)\subseteq X$ for the kernel, $R(S):=SX\subseteq Y$ for the range of $S\in L(X,Y)$. Moreover, $\id_X$ is the identity map on $X$, $\sigma(T)$ the spectrum of $T\in L(X)$ and $\abs{\cdot}$ denotes norms on finite-dimensional spaces. Finally, we write $\lip f$ for the (smallest) Lipschitz constant of a mapping $f$. 

Throughout, let $\Omega\subset\R^\kappa$ be an open, bounded (and hence Lebesgue measurable) set and $\diam\Omega:=\sup_{x,y\in\Omega}\abs{x-y}$ is its diameter. For $p\in[1,\infty)$ we introduce the space of $\K^d$-valued $p$-integrable functions
\begin{align*}
	L_d^p(\Omega)
	&:=
	\biggl\{u:\Omega\to\K^d
	\biggl|u\text{ is Lebesgue measurable with }\int_\Omega\abs{u(y)}^p\d y<\infty
	\biggr\},\\
	L_d^\infty(\Omega)
	&:=
	\biggl\{u:\Omega\to\K^d
	\biggl|u\text{ is Lebesgue measurable with }\esup_{x\in\Omega}\abs{u(y)}<\infty
	\biggr\}
\end{align*}
and equip it with $\norm{u}_p:=\intoo{\sum_{j=1}^d\int_\Omega|u_j(y)|^p\d y}^{1/p}$ resp.\ $\norm{u}_\infty:=\esup_{x\in\Omega}\abs{u(x)}$ as canonical norms. This yields a strictly decreasing scale $(L_d^p(\Omega))_{p\geq 1}$ of Banach spaces. In particular, $L_d^2(\Omega)$ is a Hilbert space with inner product $\iprod{u,v}=\sum_{j=1}^d\int_\Omega u_j(y)\overline{v_j(y)}\d y$. Finally, we write $L^p(\Omega):=L_1^p(\Omega)$. 
%
%
%
%
%
\section{Nonautonomous difference equations}
\label{sec2}
Let $\I$ be an unbounded discrete interval. 
\subsection{Hammerstein integrodifference equations}
\label{sec21}
We investigate nonautonomous integrodifference equations
\begin{align}
	\tag{$I_0$}
	u_{t+1}&=\sF_t(u_t),&
	\sF_t(u)&:=\int_\Omega k_t(\cdot,y)g_t(y,u(y))\d y
	\fall t\in\I', 
	\label{ide0}
\end{align}
whose right-hand sides are Hammerstein integral operators satisfying the assumptions: 

The \emph{kernels} $k_t:\Omega\tm\Omega\to\K^{d\tm n}$, $t\in\I'$, in \eqref{ide0} fulfill \emph{Hille-Tamarkin conditions} with exponents $p,q\in(1,\infty)$: 
\begin{enumerate}
	\item[\textbf{$(H_{q,p})$}]\quad $k_t$ is Lebesgue measurable and with $q'>1$ determined by $\tfrac{1}{q}+\tfrac{1}{q'}=1$, we assume 
	$
		\int_\Omega\biggl(\int_\Omega\abs{k_t(x,y)}^{q'}\d y\biggr)^{p/q'}\d x
		<
		\infty\fall t\in\I'.
	$
\end{enumerate}
Then \pref{propmak} guarantees that the linear integral operators
\begin{align}
	\sK_t&\in L(L_n^q(\Omega),L_d^p(\Omega)),&
	\sK_tv&:=\int_\Omega k_t(\cdot,y)v(y)\d y\fall t\in\I'
	\label{defK}
\end{align}
are well-defined and compact. 

Concerning the \emph{growth functions} $g_t:\Omega\tm\K^d\to\K^n$, $t\in\I'$, in \eqref{ide0}, for given $m\in\N_0$ and exponents $p,q\in(1,\infty)$ we assume \emph{Carath{\'e}odory conditions} for all $0\leq\ell\leq m$:
\begin{enumerate}
	\item[\textbf{$(C_{p,q}^m)$}]\quad $D_2^\ell g_t(x,\cdot):\K^d\to L_\ell(\K^d,\K^n)$ exists and is continuous for a.a.\ $x\in\Omega$, while $D_2^\ell g_t(\cdot,z):\Omega\to L_\ell(\K^d,\K^n)$ is measurable on $\Omega$ and with $mq<p$ there exist functions $c\in L^{\tfrac{pq}{p-mq}}(\Omega)$ and reals $c_0,\ldots,c_m\geq 0$ so that for a.a.~$x\in\Omega$ one has
	\begin{equation*}
		\abs{D_2^\ell g_t(x,z)}_{L_\ell(\K^d,\K^n)}
		\leq
		c(x)+c_\ell\abs{z}^{\tfrac{p-\ell q}{q}}
		\fall t\in\I',\,z\in\K^d.
	\end{equation*}
\end{enumerate}
Given this, \pref{lemmag} and \ref{lemdiff} ensure that the Nemytskii operators
\begin{align}
	\sG_t:L_d^p(\Omega)&\to L_n^q(\Omega),&
	[\sG_t(u)](x)&:=g_t(x,u(x))\fall t\in\I',\,x\in\Omega
	\label{defG}
\end{align}
are well-defined and in case $mq<p$ also $m$-times continuously differentiable. In a nutshell, the right-hand sides of \eqref{ide0} satisfy
\begin{proposition}[properties of $\sF_t$]\label{prophammer}
	Let $p,q\in(1,\infty)$ and $m\in\N$. If Hypotheses~$(H_{q,p})$ and $(C_{p,q}^m)$ hold with $mq<p$, then the right-hand sides $\sF_t:=\sK_t\sG_t:L_d^p(\Omega)\to L_d^p(\Omega)$ of \eqref{ide0} are well-defined, completely continuous and of class $C^m$ with derivatives $D^\ell(\sK_t\sG_t)(u)=\sK_tD^\ell\sG_t(u)$ for all $t\in\I'$, $0\leq\ell\leq m$ and $u\in L_d^p(\Omega)$. 
\end{proposition}
\begin{proof}
	Let $t\in\I'$. The Fredholm operators \eqref{defK} are well-defined and compact due to \pref{propmak}, while the Nemytskii operators \eqref{defG} are well-defined, bounded and continuous by \pref{lemmag}. Then \cite[pp.~25--26, Thm.~2.1]{precup:02} implies that $\sK_t\sG_t$ is completely continuous. Finally, \pref{lemdiff} and the Chain Rule applied to the $\sK_t\sG_t$ yield the claimed smoothness assertion. 
\end{proof}

Under the assumptions of \pref{prophammer} an IDE \eqref{ide0} is well-posed. This means that for arbitrary initial times $\tau\in\I$ and initial states $u_\tau\in L_d^p(\Omega)$, there are unique \emph{forward solutions}, i.e.\ sequences $(\phi_t)_{\tau\leq t}$ in $L_d^p(\Omega)$ satisfying the \emph{solution identity} $\phi_{t+1}=\sF_t(\phi_t)$ for $\tau\leq t$, $t\in\I'$. A \emph{backward solution} fulfills the solution identity for $t<\tau$ and \emph{entire solutions} $(\phi_t)_{t\in\I}$ satisfy $\phi_{t+1}\equiv\sF_t(\phi_t)$ on $\I'$. The \emph{general solution} to \eqref{ide0} is given by $\vphi^0:\set{(t,\tau,u_\tau)\in\I^2\tm L_d^p(\Omega):\,\tau\leq t}\to L_d^p(\Omega)$ via the compositions
\begin{align*}
	\vphi^0(t;\tau,u_\tau)
	&:=
	\begin{cases}
		\sF_{t-1}\circ\ldots\circ\sF_\tau(u_\tau),&\tau<t,\\
		u_\tau,&t=\tau.
	\end{cases}
\end{align*}

A subset $\cW\subseteq\I\tm L_d^p(\Omega)$ is called a \emph{nonautonomous set} and $\cW(t):=\{u\in L_d^p(\Omega):\,(t,u)\in\cW\}$, $t\in\I$, its $t$-\emph{fiber}. We say $\cW$ is \emph{forward invariant} resp.\ \emph{invariant}, provided 
$$
	\sF_t(\cW(t))\subseteq\cW(t+1)\quad\text{resp.}\quad
	\sF_t(\cW(t))=\cW(t+1)\fall t\in\I'
$$
holds. In case all fibers $\cW(t)$ are linear spaces one speaks of a \emph{linear bundle} and it is convenient to write $\cO:=\I\tm\set{0}$. 
\subsection{Linear integrodifference equations}
\label{sec22}
This section outlines the nonautonomous 'linear algebra' required to understand the dynamics of \eqref{ide0} based on linearization \cite{russ:16}. Here, if a kernel $l_t:\Omega\tm\Omega\to\K^{d\tm d}$ satisfies the Hille-Tamarkin conditions $(H_{p,p})$ with $n=d$, then a linear IDE
\begin{align}
	\tag{$L$}
	u_{t+1}&=\sL_t u_t,&
	\sL_tu&:=\int_\Omega l_t(\cdot,y)u(y)\d y
	\label{lin}
\end{align}
is well-defined with compact $\sL_t\in L(L_d^p(\Omega))$, $t\in\I'$. This yields the \emph{evolution operator}
\begin{align*}
	\Phi:\set{(t,\tau)\in\I^2:\,\tau\leq t}&\to L(L_d^p(\Omega)),&
	\Phi(t,\tau)
	&:=
	\begin{cases}
		\sL_{t-1}\cdots\sL_\tau,&\tau<t,\\
		\id_{L_d^p(\Omega)},&\tau=t.
	\end{cases}
\end{align*}

A linear IDE \eqref{lin} has an \emph{exponential dichotomy} on $\I$ \cite[p.~229, Def.~7.6.4]{henry:80}, if there exists a projection-valued sequence $(P_t)_{t\in\I}$ in $L(L_d^p(\Omega))$ and $K\geq 1$, $\alpha\in(0,1)$ so that $P_{t+1}\sL_t=\sL_tP_t$ and $\sL_t|_{N(P_t)}:N(P_t)\to N(P_{t+1})$ is an isomorphism for all $t\in\I'$ with
\begin{align*}
	\norm{\Phi(t,s)P_s}_{L(L_d^p(\Omega))}&\leq K\alpha^{t-s}\,&
	\norm{\Phi(s,t)[\id_{L_d^p(\Omega)}-P_t]}_{L_d^p(\Omega)}&\leq K\alpha^{t-s}\fall s\leq t. 
\end{align*}
This allows us to introduce nonautonomous counterparts to eigenvalue moduli in terms of the components of the \emph{dichotomy spectrum}
$$
	\Sigma_\I(\sL):=\set{\gamma>0:\,u_{t+1}=\tfrac{1}{\gamma}\sL_tu_t\text{ does not have an exponential dichotomy on }\I}
$$
for \eqref{lin}. The spectrum $\Sigma_\I(\sL)$ depends on the interval $\I$ such that $\Sigma_\I(\sL)\subseteq\Sigma_\Z(\sL)$. Under the bounded growth assumption 
\begin{equation}
	a_0:=\sup_{t\in\I'}
	\biggl(\int_\Omega\biggl(\int_\Omega\abs{l_t(x,y)}^{q'}\d y\biggr)^{\tfrac{p}{q'}}\d x\biggr)^{\tfrac{1}{p}}<\infty
	\label{bg} 
\end{equation}
it is contained in $(0,a_0]$ and there exists a $\gamma_0>0$ with $(\gamma_0,\infty)\subseteq(0,\infty)\setminus\Sigma_\I(\sL)$. As shown in \cite[Cor.~4.13]{russ:16}, $\Sigma_\I(\sL)\neq\emptyset$ is a union of at most countably many intervals which can only accumulate at some $b_\infty\geq 0$. One of the cases holds: 
\begin{enumerate}
	\item[$(S_1)$] $\Sigma_\I(\sL)$ consists of finitely many closed intervals:
	\begin{enumerate}
		\begin{figure}[ht]
			\includegraphics[scale=0.5]{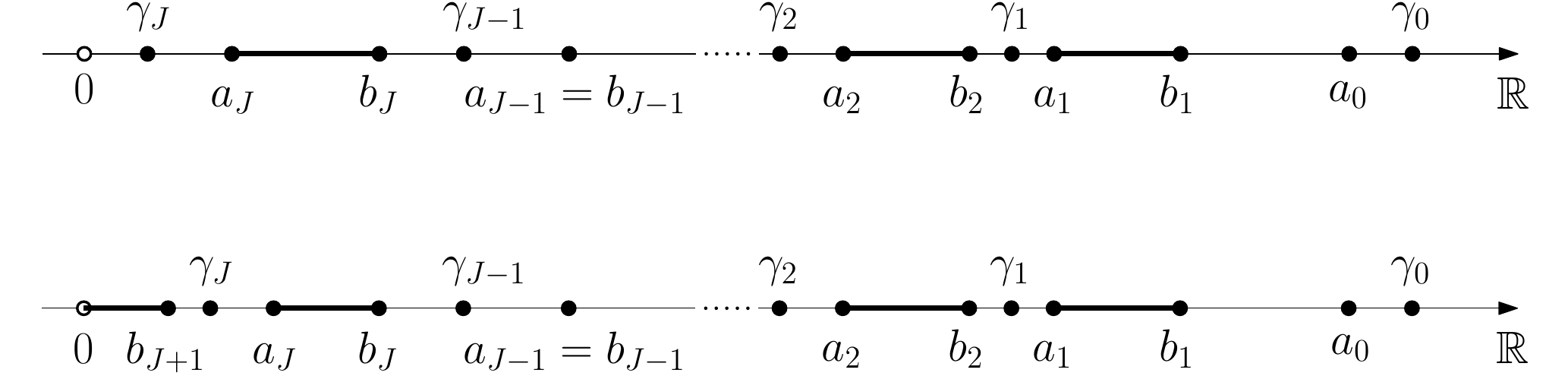}
			\caption{Case $(S_1^1)$ with $J$ compact spectral intervals (top) and $(S_1^2)$ with $J+1$ spectral intervals (bottom)}
			\label{figs11}
		\end{figure}
		\item[$(S_1^1)$] There exists a $J\in\N$ and reals $0<a_J\leq b_J<\ldots<a_1\leq b_1<a_0$ with $\Sigma_\I(\sL)=\bigcup_{j=1}^J[a_j,b_j]$. In this case we choose the $J$ rates $\gamma_j\in(b_{j+1},a_j)$, $1\leq j<J$, and $\gamma_J\in(0,a_J)$ (see \fref{figs11} (top)). 

		\item[$(S_1^2)$] There exists a $J\in\N_0$ and reals $0<b_{J+1}<a_J\leq b_J<\ldots<a_1\leq b_1<a_0$ with $\Sigma_\I(\sL)=(0,b_{J+1}]\cup\bigcup_{j=1}^J[a_j,b_j]$. Here we choose $J$ rates $\gamma_j\in(b_{j+1},a_j)$, $1\leq j\leq J$ (see \fref{figs11} (bottom)). 
	\end{enumerate}
	\begin{figure}[ht]
		\includegraphics[scale=0.5]{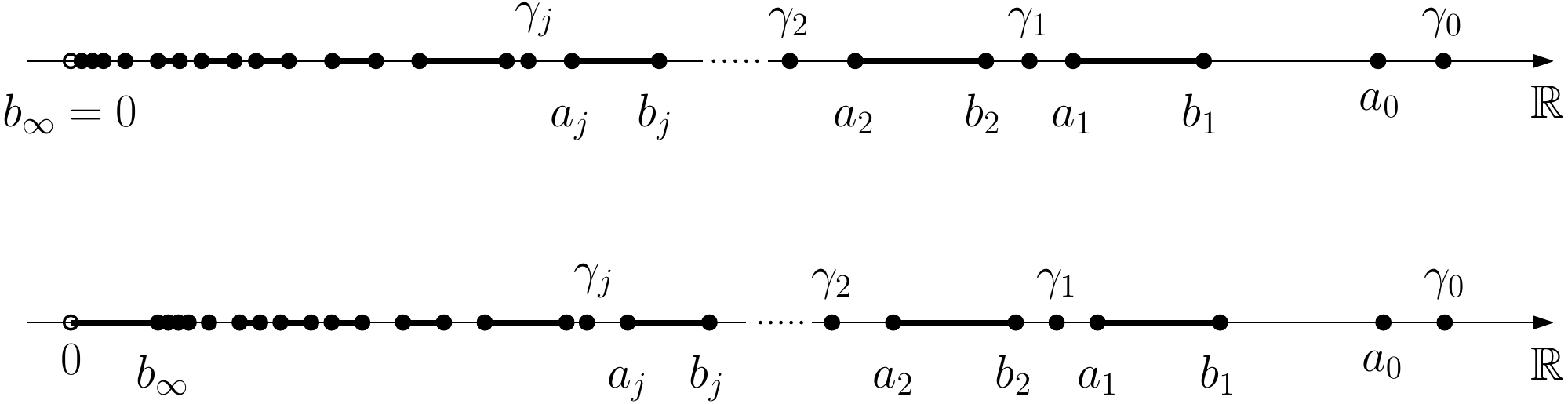}
		\caption{Case $(S_2)$ with infinitely many spectral intervals $[a_j,b_j]$ accumulating at $b_\infty=0$ i.e.\ $\sigma_\infty=\emptyset$ (top) and accumulating at $b_\infty>0$ (bottom)}
		\label{figs2}
	\end{figure}

	\item[$(S_2)$] $\Sigma_\I(\sL)$ consists of countably infinitely many intervals: There are strictly decreasing sequences $(a_j)_{j\in\N}$, $(b_j)_{j\in\N}$ in $(0,\infty)$ such that
	$
		\Sigma_\I(\sL)=\sigma_\infty\cup\bigcup_{j=1}^\infty[a_j,b_j], 
	$
	where $b_\infty<a_j\leq b_j$ for all $j\in\N$, $\lim_{j\to\infty}a_j=b_\infty$. Here, $\sigma_\infty=\emptyset$ for $b_\infty=0$ (see \fref{figs2} (top)) and otherwise $\sigma_\infty=(0,b_\infty]$ (see \fref{figs2} (bottom)). We choose countably many rates $\gamma_j\in(b_{j+1},a_j)$ for $j\in\N$. 
\end{enumerate}

Concerning nonautonomous counterparts to generalized eigenspaces, for each $\gamma>0$ we define the linear bundles
\begin{align*}
	\bar\cV_\gamma^+
	&:=
	\set{(\tau,u)\in\I\tm L_d^p(\Omega):\,\sup_{\tau\leq t}\norm{\Phi(t,\tau)u}_p\gamma^{\tau-t}<\infty},\\
	\bar\cV_\gamma^-
	&:=
	\set{(\tau,u)\in\I\tm L_d^p(\Omega):\,
	\begin{array}{r}
		\text{there exists a solution $(\phi_t)_{t\in\I}$ of \eqref{lin} so}
		\\
		\text{that $
		\sup_{\tau\leq t}\norm{\phi_t}_p\gamma^{\tau-t}<\infty$ and $\phi_\tau=u$}
	\end{array}};
\end{align*}
in case $\gamma\in(0,\infty)\setminus\Sigma_\I(\sL)$, one denotes $\bar\cV_\gamma^+$ as a $\gamma$-\emph{stable} and $\bar\cV_\gamma^-$ as a $\gamma$-\emph{unstable} bundle of \eqref{lin}. With rates $\gamma_j$ chosen as in $(S_1)$ or $(S_2)$ above, we set
\begin{align*} 
	\cV_0^+&:=\I\tm L_d^p(\Omega),&
	\cV_0^-&:=\cO,\\
	\cV_j^+&:=\bar\cV_{\gamma_j}^+,&
	\cV_j^-&:=\bar\cV_{\gamma_j}^-\quad\text{for }j>0. 
\end{align*}
Thanks to \cite[Cor.~4.13]{russ:16} the $\gamma$-unstable bundles $\bar\cV_\gamma^-$ are finite-dimensional and:
\begin{enumerate}
	\item[$(S_1)$] In both subcases the \emph{spectral bundles}
	\begin{align*}
		\cV_0&:=\cV_0^-,&
		\cV_j&:=\cV_{j-1}^+\cap\cV_j^-\neq\cO
		\fall 1\leq j\leq J
	\end{align*}
	are finite-dimensional invariant linear bundles of \eqref{lin} allowing the \emph{Whitney sum}
	$$
		\I\tm L_d^p(\Omega)=\bigoplus_{j=0}^J\cV_j\oplus\cV_J^+. 
	$$

	\item[$(S_2)$] The \emph{spectral bundles}
	\begin{align*}
		\cV_0&:=\cV_0^-,&
		\cV_j&:=\cV_{j-1}^+\cap\cV_j^-\neq\cO
		\fall j\in\N
	\end{align*}
	are finite-dimensional invariant linear bundles of \eqref{lin} with the \emph{Whitney sum}
	$$
		\I\tm L_d^p(\Omega)=\bigoplus_{j=0}^J\cV_j\oplus\cV_J^+\fall J\in\N. 
	$$
\end{enumerate}
The bundle $\cV_J^-=\bigoplus_{j=0}^J\cV_j$ satisfies $J\leq\dim\cV_J^-=\sum_{j=0}^J\dim\cV_j$. 
\section{Invariant bundles under projection methods}
\label{sec3}
Among the various numerical techniques to solve integral equations and thus to computationally simulate IDEs, we focus on projection methods \cite{atkinson:92}, \cite[pp.~446ff]{atkinson:han:00}. They are based on a sequence $\Pi_n:L_d^p(\Omega)\to X_n^d$, $n\in\N$, of bounded projections onto finite-dimen\-sional subspaces $X_n^d=X_n\tm\ldots\tm X_n\subseteq L_d^p(\Omega)$. Here each subspace $X_n$ is the span of $d_n\in\N$ linearly independent functions $\chi_1,\ldots,\chi_{d_n}:\Omega\to\R$ and it is advantageous to set $X_0:=L^p(\Omega)$. Under interpolation conditions (for collocation methods) or orthogonality conditions (for Galerkin methods), these functions determine projections $\pi_n:L^p(\Omega)\to X_n$. Given this, we set $\Pi_n:=\diag(\pi_n,\ldots,\pi_n)\in L(X_n^d)$ and refer to \sref{sec4} for concrete examples of suitable spaces $X_n$ and related projections $\pi_n$.

The resulting spatial discretizations of \eqref{ide0} are difference equations
\begin{align}
	\tag{$I_n$}
	u_{t+1}&=\sF_t^n(u_t),&
	\sF_t^n(u)&:=\Pi_n\sF_t(u)
	\label{iden}
\end{align}
being well-defined due to \pref{prophammer}. 
and $\vphi^n:\set{(t,\tau,u)\in\I^2\tm L_d^p(\Omega):\,\tau\leq t}\to L_d^p(\Omega)$ denotes their general solutions. 

Throughout the section, we suppose
\begin{hyp}
	Let $m\in\N$ and $p,q\in(1,\infty)$ with $mq<p$. Suppose that
	\begin{itemize}
		\item[(i)] $(H_{q,p})$ with a $X(\Omega)$-\emph{smoothing kernel}, that is, given a subset $X(\Omega)\subseteq L_d^p(\Omega)$ for all $t\in\I'$ the inclusions $\sK_tL_d^q(\Omega)\subseteq X(\Omega)$ hold and there exist $C_t\geq 0$ such that
		\begin{equation}
			\norm{\sK_t u}_{X(\Omega)}\leq C_t\norm{u}_q\fall u\in L_n^q(\Omega), 
			\label{nosmooth}
		\end{equation}

		\item[(ii)] $(C_{p,q}^m)$ with $c_2=\ldots=c_m=0$, $g_t(x,0)=0$ for all $t\in\I'$, a.a.\ $x\in\Omega$ and there exist measurable functions $\lambda_t,\bar\lambda_t:\Omega\to\R_+$ with
		\begin{align}
			\abs{D_2g_t(x,z)-D_2g_t(x,0)}
			&\leq
			\lambda_t(x)\quad\text{for a.a.\ }x\in\Omega\text{ and all }z\in\K^d,
			\label{no7s}\\
			\abs{g_t(x,z)-g_t(x,\bar z)}
			&\leq
			\bar\lambda_t(x)\abs{z-\bar z}\quad\text{for a.a.\ }x\in\Omega\text{ and all }z,\bar z\in\K^d
			\label{no62}
		\end{align}
		such that $\sup_{t\in\I'}C_t\biggl(\int_\Omega\bar\lambda_t(y)^{\tfrac{p}{p-q}}\d y\biggr)^{\tfrac{p-q}{p}}<\infty$ and
		\begin{align}
			L:=&\sup_{t\in\I'}
			\biggl(
			\int_\Omega\biggl(\int_\Omega\abs{k_t(x,y)}^{q'}\d y\biggr)^{\tfrac{p}{q'}}\d x
			\biggr)^{\tfrac{1}{p}}
			\biggl(\int_\Omega\lambda_t(y)^{\tfrac{p}{p-q}}\d y\biggr)^{\tfrac{p-q}{p}}
			<
			\infty,
			\label{no63}\\
			&\sup_{t\in\I'}
			\biggl(
			\int_\Omega\biggl(\int_\Omega\abs{k_t(x,y)}^{q'}\d y\biggr)^{\tfrac{p}{q'}}\d x
			\biggr)^{\tfrac{1}{p}}
			\biggl(\int_\Omega\bar\lambda_t(y)^{\tfrac{p}{p-q}}\d y\biggr)^{\tfrac{p-q}{p}}
			<
			\infty,
			\label{no64}
		\end{align}

		\item[(iii)] there exists a function $\Gamma:\R_+\to(0,\infty)$ with $\lim_{\varrho\searrow 0}\Gamma(\varrho)=0$ such that the \emph{discretization error} satisfies
		\begin{equation}
			\norm{u-\Pi_nu}_p
			\leq
			\Gamma(\tfrac{1}{n})\norm{u}_{X(\Omega)}\fall n\in\N,\,u\in X(\Omega). 
			\label{thmpifb4}
		\end{equation}
	\end{itemize}
\end{hyp}
The subsequent \sref{sec4} is devoted to specific error estimates \eqref{thmpifb4} for various function spaces $X(\Omega)$. Moreover, in order to illustrate the smoothing property (i) we consider two nonsmooth kernels depending on dispersal rates $\delta_t>0$, $t\in\I'$: 
\begin{example}[Laplace kernel]
	Consider the \emph{Laplace kernel} (see \cite[pp.~18ff]{lutscher:19})
	\begin{align*}
		k_t:\Omega^2&\to(0,\infty),&
		k_t(x,y):=\tfrac{\delta_t}{2}e^{-\delta_t\abs{x-y}}
	\end{align*}
	acting on a habitat $\Omega=(-\tfrac{l}{2},\tfrac{l}{2})$, $l>0$. As a continuous, globally bounded function it satisfies $(H_{q,p})$ with arbitrary exponents $p,q\in(1,\infty)$ and $d=n=1$. In particular, 
	\begin{align*}
		\int_\Omega\intoo{\int_\Omega\abs{k_t(x,y)}^{q'}\d y}^{p/q'}\d x
		\leq
		\int_\Omega\intoo{\int_\Omega\intoo{\tfrac{\delta_t}{2}}^{q'}\d y}^{p/q'}\d x
		=
		l^{1+p-1/q}\intoo{\tfrac{\delta_t}{2}}^p
	\end{align*}
	and therefore $\norm{\sK_t}_{L(L^q(\Omega),L^p(\Omega))}\leq l^{1+p-1/q}\intoo{\tfrac{\delta_t}{2}}^p$ for all $t\in\I'$. In addition, 
	$$
		(\sK_t u)'=\sK_t^{(1)}u
		\quad\text{with}\quad
		(\sK_t^{(1)}u)(x):=
		-\frac{\delta_t^2}{2}
		\int_\Omega\sgn(x-y)e^{-\delta_t\abs{x-y}}u(y)\d y
	$$
	for all $x\in\Omega$ and $u\in L^q(\Omega)$, as well as $\|\sK_t^{(1)}\|_{L(L^q(\Omega),L^p(\Omega))}\leq l^{1+p-1/q}\intoo{\tfrac{\delta_t^2}{2}}^p$. Hence, we can choose the smaller space $X(\Omega):=W^{1,p}(\Omega)$ and obtain
	\begin{align*}
		\norm{\sK_t u}_{W^{1,p}(\Omega)}
		&=
		\intoo{\norm{\sK_tu}_p^p+\|\sK_t^{(1)}u\|_p^p}^{1/p}
		\leq
		C_t\norm{u}_q
		\fall u\in L^q(\Omega),
	\end{align*}
	with the constants $C_t:=l^{1+1/p-1/(pq)}\tfrac{\delta_t}{2}(1+(\tfrac{\delta_t}{2})^p)^{1/p}$. In conclusion, the Laplace kernel is $W^{1,p}(\Omega)$-smoothing and the estimate \eqref{nosmooth} holds. 
\end{example}
\begin{example}[root kernel]
	The \emph{exponential root kernel} (see \cite[p.~72 for $\alpha=\tfrac{1}{2}$]{lutscher:19})
	\begin{align*}
		k_t:\Omega^2&\to(0,\infty),&
		k_t(x,y):=\tfrac{\delta_t^2}{2}e^{-\delta_t\abs{x-y}^\alpha}
	\end{align*}
	with exponent $\alpha\in(0,1]$ may be defined on an open, bounded habitat $\Omega\subset\R^\kappa$. Being continuous and globally bounded, it again fulfills $(H_{q,p})$ with arbitrary $p,q\in(1,\infty)$ and $d=n=1$. Given $u\in L^q(\Omega)$ the H\"older inequality implies
	$$
		\abs{\sK_tu(x)}
		\stackrel{\eqref{defK}}{=}
		\abs{\int_\Omega k_t(x,y)u(y)\d y}
		\leq
		\intoo{\int_\Omega\abs{k_t(x,y)}^{q'}\d y}^{1/q'}\norm{u}_q
		\leq
		\frac{\delta_t^2}{2}\lambda_\kappa(\Omega)^{1/q'}\norm{u}_q
	$$
	for all $x\in\Omega$, as well as using the Mean Value Estimate (applied to $x\mapsto e^{-x}$)
	\begin{align*}
		\abs{\sK_tu(x)-\sK_tu(\bar x)}
		&\stackrel{\eqref{defK}}{=}
		\abs{\int_\Omega[k_t(x,y)-k_t(\bar x,y)]u(y)\d y}\\
		&\leq
		\intoo{\int_\Omega\abs{k_t(x,y)-k_t(\bar x,y)}^{q'}\d y}^{1/q'}\norm{u}_q\\
		&\leq
		\frac{\delta_t^2}{2}
		\intoo{\int_\Omega\abs{e^{-\delta_t\abs{x-y}^\alpha}-e^{-\delta_t\abs{\bar x-y}^\alpha}}^{q'}\d y}^{1/q'}\norm{u}_q\\
		&\leq
		\frac{\delta_t^3}{2}
		\intoo{\int_\Omega\abs{\abs{x-y}^\alpha-\abs{\bar x-y}^\alpha}^{q'}\d y}^{1/q'}\norm{u}_q\\
		&\leq
		\frac{\delta_t^3}{2}
		\intoo{\int_\Omega\abs{\abs{x-y}-\abs{\bar x-y}}^{\alpha q'}\d y}^{1/q'}\norm{u}_q\\
		&\leq
		\frac{\delta_t^3}{2}
		\intoo{\int_\Omega\abs{x-\bar x}^{\alpha q'}\d y}^{1/q'}\norm{u}_q\\
		&=
		\frac{\delta_t^3}{2}\lambda_\kappa(\Omega)^{1/q'}\abs{x-\bar x}^\alpha\norm{u}_q\fall x,\bar x\in\Omega. 
	\end{align*}
	Thus, $\sK_tu$ is H\"older with exponent $\alpha$ and we can choose $X(\Omega)=C^\alpha(\overline{\Omega})$ with
	$$
		\norm{\sK_tu}_{C^\alpha(\overline{\Omega})}
		\leq
		\tfrac{\delta_t^2}{2}\lambda_\kappa(\Omega)^{1/q'}\max\set{1,\delta_t}\norm{u}_q
		\fall u\in L^q(\Omega).
	$$
	In summary, the exponential root kernel is $C^\alpha(\overline{\Omega})$-smoothing and the estimate \eqref{nosmooth} holds with the constant $C_t:=\tfrac{\delta_t^2}{2}\lambda_\kappa(\Omega)^{1/q'}\max\set{1,\delta_t}$ for all $t\in\I'$. 
\end{example}
\subsection{Pseudo-stable and -unstable bundles}
Under the above Hypotheses (i--iii), the right-hand side of \eqref{ide0} is at least continuously differentiable (cf.~\pref{prophammer}). Its variational equation along the trivial solution
\begin{align}
	\tag{$V_0$}
	v_{t+1}&=D\sF_t(0)v_t,&
	D\sF_t(0)v&:=\int_\Omega k_t(\cdot,y)D_2g_t(y,0)v(y)\d y
	\label{var0}
\end{align}
is a linear IDE of the form \eqref{lin} with kernel $l_t(x,y):=k_t(x,y)D_2g_t(y,0)\in\K^{d\tm d}$, whose dichotomy spectrum will be denoted by $\Sigma_\I$. 
\begin{figure}[ht]
	\includegraphics[scale=0.5]{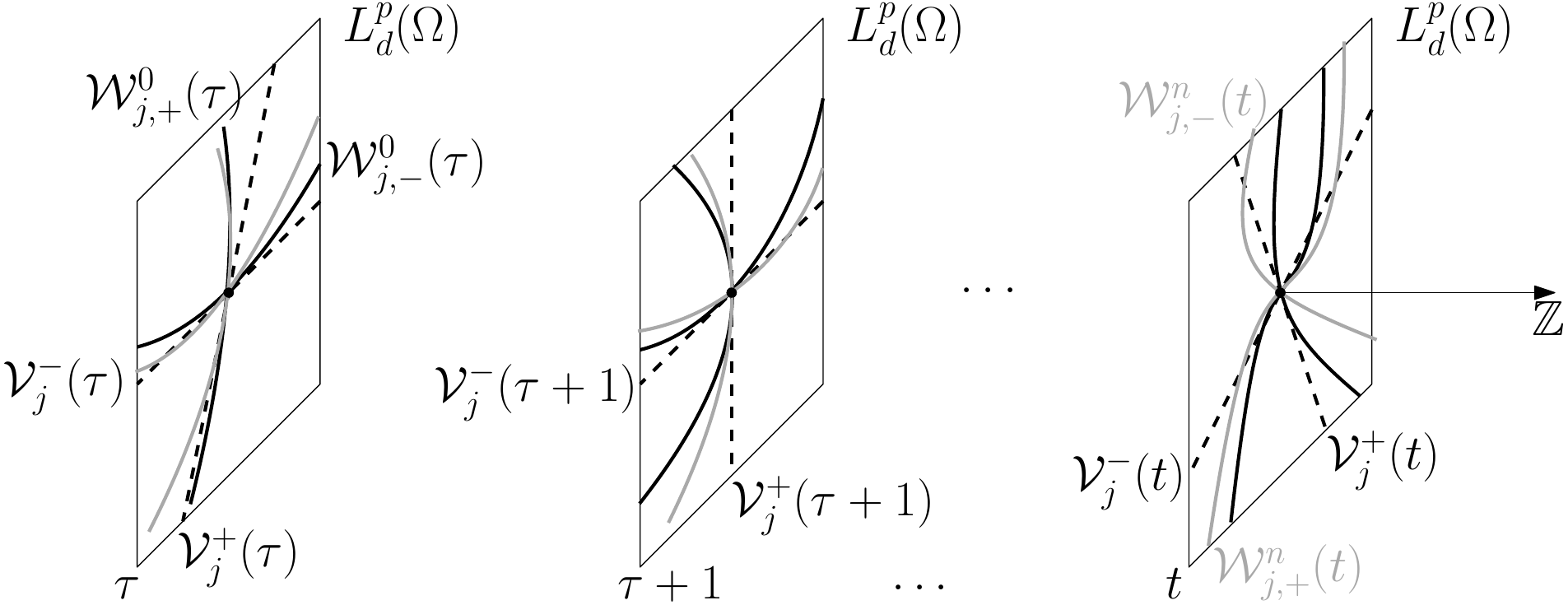}
	\caption{The pseudo-stable bundle $\cW_{j,+}^0$ and the pseudo-un\-stable bundle $\cW_{j,-}^0$ (black) of \eqref{ide0} intersect along the trivial solution and perturb as bundles $\cW_{j,+}^n$ and $\cW_{j,-}^n$, $n\geq N$ (grey) for the spatial discretizations \eqref{iden}. Both consist of $t$-fibers being graphs over $\cV_j^+$ resp.\ $\cV_j^-$.}
	\label{fig3}
\end{figure}
\begin{theorem}[pseudo-stable and -unstable bundles under discretization]\label{thmpifb}
	Let Hypotheses (i--iii) hold, $\gamma_j>0$, $j\in\N_0$, be as in $(S_1)$ or $(S_2)$, choose $\alpha_j<\beta_j$ so that
	\begin{align*}
		\alpha_j&<\gamma_j<\beta_j,&
		(\alpha_j,\beta_j)\cap\Sigma_\I&=\emptyset
	\end{align*}
	and let $\cV_j^+,\cV_j^-$ denote the associated linear bundles of the variational equation \eqref{var0}. Then there exists a $K\geq 1$ depending only on \eqref{var0}, so that if 
	\begin{align}
		L&<\frac{\delta_{\max}}{4K},&
		\delta_{\max}&:=\frac{\beta_j-\alpha_j}{2}, 
		\label{nosmall}
	\end{align}
	$\delta\in(4KL,\delta_{\max}]$ is fixed and $\Gamma_j:=[\alpha_j+\delta,\beta_j-\delta]$, then there exists a $N\in\N$ such that the following statements are true for $n=0$ and $n\geq N$ (see \fref{fig3}): 
	\begin{enumerate}
		\item[(a)] In case $\I$ is unbounded above, then the \emph{pseudo-stable bundle}
		$$
			\cW_{j,+}^n
			:=
			\set{(\tau,u)\in\I\tm L_d^p(\Omega):\,
			\sup_{\tau\leq t}\gamma^{\tau-t}\norm{\vphi^n(t;\tau,u)}_p<\infty}
		$$
		is a forward invariant bundle of \eqref{iden}, which is independent of $\gamma\in\Gamma_j$ and allows the representation
		$$
			\cW_{j,+}^n
			=
			\set{(\tau,v+w_{j,+}^n(\tau,v))\in\I\tm L_d^p(\Omega):\,(\tau,v)\in\cV_j^+}
		$$
		with continuous mappings $w_{j,+}^n:\cV_j^+\to L_d^p(\Omega)$ satisfying $w_{j,+}^n(\tau,0)=0$ and
		\begin{itemize}
			\item[$(a_1)$] there exist $C_0^+(n),N_0^+>0$ with $\lim_{n\to\infty}C_0^+(n)=\tfrac{K^2L}{\delta-2KL}=:C_0^+(0)$ and for all $(\tau,v)\in\cV_j^+$ holds
			\begin{align*}
				\lip w_{j,+}^n(\tau,\cdot)&\leq C_0^+(n),&
				\norm{w_{j,+}^n(\tau,v)-w_{j,+}^0(\tau,v)}_p
				&\leq
				N_0^+\Gamma(\tfrac{1}{n})\norm{v}_p, 
			\end{align*}

			\item[$(a_2)$] if $\alpha_j^{m_+}<\beta_j$ and $\delta_{\max}:=\min\Bigl\{\tfrac{\beta_j-\alpha_j}{2},\alpha_j\bigl(\sqrt[{m_+}]{\frac{\alpha_j+\beta_j}{\alpha_j+\alpha_j^{m_+}}}-1\bigr)\Bigr\}$, then the derivatives $D_2^\ell w_{j,+}^n:\cV_j^+\to L_\ell(L_d^p(\Omega))$ exist up to order $\ell\leq m_+\leq m$ as continuous maps and for $1\leq\ell<m_+$ there are $N_\ell^+>0$ such that
			\begin{align*}
				\hspace*{-10mm}
				\norm{D_2^\ell w_{j,+}^n(\tau,v)-D_2^\ell w_{j,+}^0(\tau,v)}_{L_\ell(L_d^p(\Omega))}
				&\leq
				N_\ell^+\Gamma(\tfrac{1}{n})\norm{v}_p
				\fall(\tau,v)\in\cV_j^+. 
			\end{align*}
		\end{itemize}

		\item[(b)] In case $\I$ is unbounded below, then the \emph{pseudo-unstable bundle}
		$$
			\cW_{j,-}^n
			:=
			\set{(\tau,u)\in\I\tm L_d^p(\Omega):
			\begin{array}{l}
			\text{there exists a solution }(\phi_t)_{t\in\I}\text{ of \eqref{iden} so}\\
			\text{that }\sup_{t\leq\tau}\gamma^{\tau-t}\norm{\phi_t}_p<\infty\text{ and }\phi_\tau=u
			\end{array}
			}
		$$
		is a finite-dimensional (unless for spectra $(S_1^1)$ and $j=J$) invariant bundle of \eqref{iden}, which is independent of $\gamma\in\Gamma_j$ and allows the representation
		$$
			\cW_{j,-}^n
			=
			\set{(\tau,v+w_{j,-}^n(\tau,v))\in\I\tm L_d^p(\Omega):\,(\tau,v)\in\cV_j^-}
			\subseteq
			\begin{cases}
				\I\tm X(\Omega),&n=0,\\
				\I\tm X_n^d,&n\geq N
			\end{cases}
		$$
		with continuous mappings $w_{j,-}^n:\cV_j^-\to L_d^p(\Omega)$ satisfying $w_{j,-}^n(\tau,0)=0$ and
		\begin{itemize}
			\item[$(b_1)$] there exist $C_0^-(n),N_0^->0$ with $\lim_{n\to\infty}C_0^-(n)=\tfrac{K^2L}{\delta-2KL}=:C_0^-(0)$ and for all $(\tau,v)\in\cV_j^-$ holds
			\begin{align*}
				\lip w_{j,-}^n(\tau,\cdot)&\leq C_0^-(n),&
				\norm{w_{j,-}^n(\tau,v)-w_{j,-}^0(\tau,v)}_p
				&\leq
				N_0^-
				\Gamma(\tfrac{1}{n})\norm{v}_p, 
			\end{align*}

			\item[$(b_2)$] if $\alpha_j<\beta_j^{m_-}$ and $\delta_{\max}:=\min\Bigl\{\tfrac{\beta_j-\alpha_j}{2},\beta_j\bigl(1-\sqrt[{m_-}]{\frac{\alpha_j+\beta_j}{\beta_j+\beta_j^{m_-}}}\bigr)\Bigr\}$, then the derivatives $D_2^\ell w_{j,-}^n:\cV_j^-\to L_\ell(L_d^p(\Omega))$ exist up to order $\ell\leq m_-\leq m$ as continuous maps and for $1\leq\ell<m_-$ there are $N_\ell^->0$ such that
			\begin{align*}
				\hspace*{-10mm}
				\norm{D_2^\ell w_{j,-}^n(\tau,v)-D_2^\ell w_{j,-}^0(\tau,v)}_{L_\ell(L_d^p(\Omega))}
				&\leq
				N_\ell^-\Gamma(\tfrac{1}{n})\norm{v}_p
				\fall(\tau,v)\in\cV_j^-. 
			\end{align*}
		\end{itemize}

		\item[(c)] If $\I=\Z$ and $L<\frac{\delta}{6K}$, then $\cW_{j,+}^n\cap\cW_{j,-}^n=\cO$ and the zero solution is the only entire solution $\phi$ to \eqref{iden} satisfying $\sup_{t\in\Z}\gamma^{-t}\norm{\phi_t}_p<\infty$ for all $\gamma\in\Gamma_j$. 
	\end{enumerate}
\end{theorem}
The discretization bound $N\in\N$ under which the bundles $\cW_{j,+}^0,\cW_{j,-}^0$ of \eqref{ide0} persist for $n\geq N$ is large, provided the spectral gap $\beta_j-\alpha_j$ is small, the habitat $\Omega\subset\R^\kappa$ is large (in terms of its diameter), the Lipschitz constants $\bar\lambda_t$ or the `operator norms' in \eqref{no63} are large, or the convergence rate via $\Gamma$ is small. 

In case $j=0$ the bundles constructed in \tref{thmpifb} reduce to $\cW_{0,+}^n=\I\tm L_d^p(\Omega)$ and $\cW_{0,-}^n=\cO$. Further specific mention deserves a dichotomy spectrum of the form $(S_1^1)$ and $j=J$, where we have $\cW_{J,+}^n=\cO$ and $\cW_{J,-}^n=\I\tm X_n^d$ ($n\geq N$) or $X(\Omega)$ ($n=0$). 

\begin{figure}[ht]
	\includegraphics[scale=0.5]{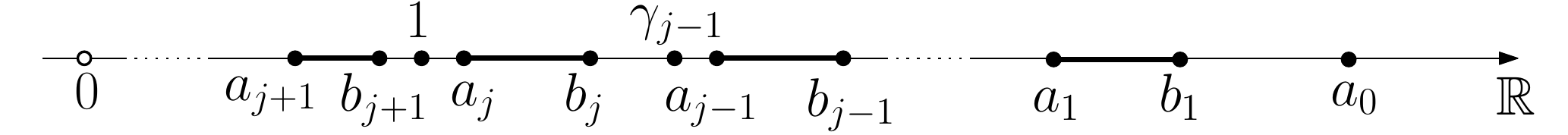}
	\caption{Hyperbolic situation $1\not\in\Sigma_\Z$ yielding the stable bundles $\cW_{j,+}^n$ and the unstable bundles $\cW_{j,-}^n$}
	\label{fig4}
\end{figure}
\begin{remark}[hyperbolic case]
	The trivial solution of \eqref{ide0} is called \emph{hyperbolic}, if $1\not\in\Sigma_\Z$. In this case we assume $1\in(b_{j+1},a_j)$ for some $j\in\N$ (see \fref{fig4}). 

	(1) Then $\cW_s:=\cW_{j,+}^0$ is called \emph{stable} and $\cW_u:=\cW_{j,-}^0$ \emph{unstable bundle} for \eqref{ide0}. The stable bundle contains the \emph{strongly stable bundles} $\cW_{i,+}^0$, $i>j$, while the unstable bundle contains the finitely many \emph{strongly unstable bundles} $\cW_{i,-}^0$, $0<i<j$, of \eqref{ide0}.
	
	(2) The dichotomy spectrum $\Sigma_\Z$ of \eqref{var0} behaves upper-semi\-con\-tinuously under discretization, that is, if $\Sigma_\Z^n$ denotes the spectrum of the discretized variational equation
	\begin{equation}
		\tag{$V_n$}
		v_{t+1}=\Pi_nD\sF_t(0)v_t,
		\label{varn}
	\end{equation}
	then for each $\eps>0$ there exists a $\bar N\in\N$ such that $\Sigma_\Z^n\subseteq B_\eps(\Sigma_\Z^0)$ for all $n\geq\bar N$. This guarantees that hyperbolicity is preserved under discretization. In particular, $\cW_{i,+}^0$ persist as (strongly) stable bundles $\cW_{i,+}^n$, $j\leq i$, and $\cW_{j,-}^0$ persists as (strongly) unstable bundles $\cW_{j,-}^n$, $0<i\leq j$, of the discretizations \eqref{iden} for $n\geq\max\set{N,\bar N}$. 
\end{remark}
\begin{remark}[periodic IDEs]
	The periodic situation $k_t=k_{t+\theta}:\Omega\tm\Omega\to\K^{d\tm n}$ and $g_t=g_{t+\theta}:\Omega\tm\K^d\to\K^n$ for all $t\in\Z$ with some $\theta\in\N$ is a relevant special case of general nonautonomous problems \eqref{ide0}. Here the dichotomy spectrum is
	$$
		\Sigma_{\Z}=\sqrt[\theta]{\abs{\sigma(D\sF_\theta(0)\cdots D\sF_1(0))}}\setminus\set{0}
	$$
	and the $\theta$-periodicity of \eqref{ide0} extends to the spatial discretizations \eqref{iden}, as well as to the bundles $\cW_{j,+}^n$, $\cW_{j,-}^n$ for $n=0$ and $n\geq N$. For the fibers this means $\cW_{j,+}^n(t+\theta)=\cW_{j,+}^n(t)$ and $\cW_{j,-}^n(t+\theta)=\cW_{j,-}^n(t)$ for all $t\in\Z$. 
\end{remark}
\begin{proof}[Proof of \tref{thmpifb}]
	Due to \pref{prophammer} the right-hand side of \eqref{ide0} is of class $C^m$ with 
	$$
		\sL_tv:=D\sF_t(0)v=\int_\Omega k_t(\cdot,y)D_2g_t(y,0)v(y)\d y\fall t\in\I',\,v\in L_d^p(\Omega)
	$$
	as derivative along the trivial solution. It results from the Lipschitz condition \eqref{no62} and \cite[p.~363, Prop.~C.1.1]{poetzsche:10} that $\abs{D_2g_t(y,0)}\leq\bar\lambda_t(y)$ for all $t\in\I'$ and a.a.\ $y\in\Omega$. We derive from \eqref{no64} that the variational equation \eqref{var0} satisfies \eqref{bg} (with kernel $l_t(x,y):=k_t(x,y)D_2g_t(y,0)$) and $\Sigma_\I\subseteq(0,a_0]$; so the spectrum is of the form required in \sref{sec22}. 
	
	In order to apply the perturbation \tref{thmifb} in $X=L_d^p(\Omega)$ we first specify the general difference equation \eqref{fbdeq} as 
	\begin{equation}
		u_{t+1}
		=
		\sF_t(u_t)+\tfrac{\theta}{\Gamma(\tfrac{1}{n})}\intcc{\sF_t^n(u_t)-\sF_t(u_t)}
		=
		\sL_tu_t+\sN_t(u_t)+\theta\bar\sN_t(u_t)
		\label{eqtheta}
	\end{equation}
	depending on a real parameter $\theta$ and with the functions $\sN_t,\bar\sN_t:L_d^p(\Omega)\to L_d^p(\Omega)$, 
	\begin{align*}
		\sN_t(u)&:=\sF_t(u)-D\sF_t(0)u=\int_\Omega k_t(\cdot,y)h_t(y,u(y))\d y,\\
		\bar\sN_t(u)&:=\frac{1}{\Gamma(\tfrac{1}{n})}[\Pi_n-\id_{L_d^p(\Omega)}]\sF_t(u)
		=
		\frac{1}{\Gamma(\tfrac{1}{n})}\intcc{\Pi_n-\id_{L_d^p(\Omega)}}\int_\Omega k_t(\cdot,y)g_t(y,u(y))\d y,
	\end{align*}
	where $h_t(y,z):=g_t(y,z)-D_2g_t(y,0)z$. For $\theta=0$ the semilinear difference equation \eqref{eqtheta} reduces to the original problem \eqref{ide0}, while for $\theta=\Gamma(\tfrac{1}{n})$ one obtains the spatial discretization \eqref{iden}. We next verify the assumptions of \tref{thmifb}. 

	\emph{ad $(H1)$}: 
	According to the choice of $\alpha_j<\beta_j$, we obtain that the variational equation \eqref{var0} along the trivial solution satisfies the dichotomy estimates \eqref{noh1} with real constants $K\geq 1$, $\alpha=\alpha_j$, $\beta=\beta_j$ and an invariant projector $P_t\in L(L_d^p(\Omega))$, $t\in\I$. By the choice of $\gamma\in\Gamma_j$ one has $R(P_t)=\cV_j^+(t)$ (if $\I$ is unbounded above) and $N(P_t)=\cV_j^-(t)$ (if $\I$ is unbounded below) for all $t\in\I$. 

	\emph{ad $(H2)$}: From Hypothesis (ii) it immediately results that both $\sN_t$ and $\bar\sN_t$ vanish identically in $0$. Using the Mean Value Theorem we have for all $t\in\I'$ that
	\begin{align*}
		\abs{h_t(x,z)-h_t(x,\bar z)}
		&=
		\abs{g_t(x,z)-g_t(x,\bar z)-D_2g_t(x,0)(z-\bar z)}\\
		&=
		\abs{\int_0^1D_2g_t(x,\bar z+\vartheta(z-\bar z))-D_2g_t(x,0)\d\vartheta(z-\bar z)}\\
		&\leq
		\int_0^1\abs{D_2g_t(x,\bar z+\vartheta(z-\bar z))-D_2g_t(x,0)}\d\vartheta\abs{z-\bar z}
		\stackrel{\eqref{no7s}}{\leq}
		\lambda_t(x)\abs{z-\bar z}
	\end{align*}
	for a.a.\ $x\in\Omega$ and $z,\bar z\in\K^d$. Hence, if $\sH_t:L_d^p(\Omega)\to L_n^q(\Omega)$ denotes the Nemytskii operator induced by $h_t$, then \cref{cormag} implies for $t\in\I'$ that $\sN_t=\sK_t\sH_t$ and then
	\begin{align*}
		\norm{\sN_t(u)-\sN_t(\bar u)}_p
		& \leq
		\norm{\sK_t}_{L(L_n^q(\Omega),L_d^p(\Omega))}\norm{\sH_t(u)-\sH_t(\bar u)}_q\\
		& \stackrel{\eqref{propmak1}}{\leq}
		\biggl(\int_\Omega\biggl(\int_\Omega\abs{k_t(x,y)}^{q'}\d y\biggr)^{\tfrac{p}{q'}}\d x\biggr)^{\tfrac{1}{p}}
		\norm{\sH_t(u)-\sH_t(\bar u)}_q\\
		& \stackrel{\eqref{cormag2}}{\leq}
		\biggl(\int_\Omega\biggl(\int_\Omega\abs{k_t(x,y)}^{q'}\d y\biggr)^{\tfrac{p}{q'}}\d x\biggr)^{\tfrac{1}{p}}
		\biggl(\int_\Omega\lambda_t(y)^{\tfrac{p}{p-q}}\d y\biggr)^{\tfrac{p-q}{p}}
		\norm{u-\bar u}_p\\
		& \stackrel{\eqref{no63}}{\leq}
		L\norm{u-\bar u}_p
		\fall u,\bar u\in L_d^p(\Omega).
	\end{align*}
	Concerning the nonlinearities $\bar\sN_t$ we obtain from the $X(\Omega)$-smoothing property assumed in Hypothesis (i) that $\sF_t=\sK_t\sG_t$ maps into $X(\Omega)$ and thus for $t\in\I'$ that
	\begin{align*}
		\norm{\bar\sN_t(u)-\bar\sN_t(\bar u)}_p
		&=
		\tfrac{1}{\Gamma(\tfrac{1}{n})}\norm{[\Pi_n-\id_{L_d^p(\Omega)}][\sF_t(u)-\sF_t(\bar u)]}_p
		\stackrel{\eqref{thmpifb4}}{\leq}
		\norm{\sF_t(u)-\sF_t(\bar u)}_{X(\Omega)}\\
		&\stackrel{\eqref{nosmooth}}{\leq}
		C_t\norm{\sG_t(u)-\sG_t(\bar u)}_q
		\stackrel{\eqref{cormag2}}{\leq}
		C_t\biggl(\int_\Omega\bar\lambda_t(y)^{\tfrac{p}{p-q}}\d y\biggr)^{\tfrac{p-q}{p}}
		\norm{u-\bar u}_p
	\end{align*}
	for all $u,\bar u\in L_d^p(\Omega)$, yielding $(H2)$ with $\bar L:=\sup_{t\in\I'}C_t\biggl(\int_\Omega\bar\lambda_t(y)^{\tfrac{p}{p-q}}\d y\biggr)^{\tfrac{p-q}{p}}<\infty$. We now choose $N\in\N$ so large that $\bar L\Gamma(\tfrac{1}{n})\leq L$ holds for all $n\geq N$, which ensures the inclusion $\Gamma(\tfrac{1}{n})\in\Theta:=\set{\theta\in\R:\,\bar L\abs{\theta}\leq L}$. 

	\emph{ad $(H3)$}: From the Chain Rule, for $1<\ell\leq m$ we directly conclude the derivatives
	\begin{align*}
		D^\ell\sN_t(u)
		&=
		\sK_tD^\ell\sG_t(u),&
		D^\ell\bar\sN_t(u)
		&=
		\tfrac{1}{\Gamma(\tfrac{1}{n})}[\Pi_n-\id_{L_d^p(\Omega)}]
		\sK_tD^\ell\sG_t(u)
		\fall t\in\I'
	\end{align*}
	and $u\in L_d^p(\Omega)$. If $q_\ell:=\tfrac{pq}{p-\ell q}$, then it is a consequence of \pref{lemdiff}(b) and 
	\begin{align*}
		\norm{D^\ell\sN_t(u)}_{L_\ell(L_d^p(\Omega))}
		&\leq
		\norm{\sK_t}_{L(L_n^q(\Omega),L_d^p(\Omega))}\norm{D^\ell\sG_t(u)}_{L_\ell(L_d^p(\Omega),L_n^q(\Omega))}\\
		&\leq
		\norm{\sK_t}_{L(L_n^q(\Omega),L_d^p(\Omega))}\intoo{\norm{c}_{q_\ell}+c_\ell\norm{u}_p^{q/q_\ell}}
	\end{align*}
	combined with $c_2=\ldots=c_m=0$ that the derivatives $D^\ell\sN_t$ are globally bounded (uniformly in $t\in\I'$). Similarly, for all $t\in\I'$ and $h_1,\ldots,h_\ell\in L_d^p(\Omega)$ one has
	\begin{align*}
		\norm{D^\ell\bar\sN_t(u)h_1\cdots h_\ell}_p
		&=
		\tfrac{1}{\Gamma(\tfrac{1}{n})}\norm{[\Pi_n-\id_{L_d^p(\Omega)}]D^\ell\sF_t(u)h_1\cdots h_\ell}_p\\
		&\stackrel{\eqref{thmpifb4}}{\leq}
		\norm{\sK_tD^\ell\sG_t(u)h_1\cdots h_\ell}_{X(\Omega)}
		\stackrel{\eqref{nosmooth}}{\leq}
		C_t\norm{D^\ell\sG_t(u)h_1\cdots h_\ell}_q\\
		&\leq
		C_t\intoo{\norm{c}_{q_\ell}+c_\ell\norm{u}_p^{q/q_\ell}}
		\norm{h_1}_p\cdots\norm{h_\ell}_p
		\fall 1<\ell\leq m
	\end{align*}
	and $u\in L_d^p(\Omega)$, which also yields the global boundedness of $D^\ell\bar\sN_t$ (uniformly in $t\in\I'$). After these preparations we can address the proof of statements (a--c): 
	
	(a) From \tref{thmifb}(a) we obtain that \eqref{eqtheta} has a $\gamma$-stable bundle $\cW^+(\theta)\subseteq\I\tm L_d^p(\Omega)$ as graph of a continuous function $w^+(\cdot;\theta):\I\tm L_d^p(\Omega)\to L_d^p(\Omega)$ for all $\theta\in\Theta$. Then 
	\begin{align*}
		\cW_{j,+}^0&:=\cW^+(0),&
		w_{j,+}^0&:=w^+(\cdot;0)|_{\cV_j^+},\\
		\cW_{j,+}^n&:=\cW^+(\Gamma(\tfrac{1}{n})),&
		w_{j,+}^n&:=w^+(\cdot;\Gamma(\tfrac{1}{n}))|_{\cV_j^+}\fall n\geq N,
	\end{align*}
	yield the $\gamma$-stable bundles of \eqref{ide0} and \eqref{iden} resp.\ the mappings parametrizing them; in particular, $w_{j,+}^n(\tau,0)\equiv 0$ on $\I$. \tref{thmifb}$(a_2)$ guarantees the Lipschitz conditions
	\begin{align*}
		\lip w_{j,+}^0(\tau,\cdot)
		&\stackrel{\eqref{noA3}}{\leq}
		\frac{K^2L}{\delta-2KL},&
		\lip w_{j,+}^n(\tau,\cdot)
		&\stackrel{\eqref{noA3}}{\leq}
		\frac{K^2(L+\Gamma(\tfrac{1}{n})\bar L)}{\delta-K(L+\Gamma(\tfrac{1}{n})\bar L)}
		\fall\tau\in\I,\,n\geq N
	\end{align*}
	and defining $C_0^+(0):=\tfrac{K^2L}{\delta-2KL}$, $C_0^+(n):=\tfrac{K^2(L+\Gamma(\tfrac{1}{n})\bar L)}{\delta-K(L+\Gamma(\tfrac{1}{n})\bar L)}$ yields $\lim_{n\to\infty}C_0^+(n)=C_0^+(0)$. Furthermore, the Lipschitz estimate \eqref{noA4} leads to 
	\begin{align*}
		\norm{w_{j,+}^n(\tau,v)-w_{j,+}^0(\tau,v)}_p
		&=
		\norm{w^+(\tau,v;\Gamma(\tfrac{1}{n}))-w^+(\tau,v;0)}_p\\
		&\leq
		\frac{2\delta K^3\bar L}{(\delta-4K\bar L)^2}\Gamma(\tfrac{1}{n})\norm{v}_p
		\fall(\tau,v)\in\cV_j^+,\,n\geq N.
	\end{align*}
	If we set $N_0^+:=\tfrac{2\delta K^3\bar L}{(\delta-4K\bar L)^2}$, then this establishes the assertion $(a_1)$. Under the additional assumptions imposed for $(a_2)$ we obtain from \tref{thmifb}$(a_3)$ that the partial derivatives $D_3D_2^\ell w_+$ exist for $0\leq\ell<m_+$. The Mean Value Estimate implies
	\begin{align*}
		\norm{D_2^\ell w_{j,+}^n(\tau,v)-D_2^\ell w_{j,+}^0(\tau,v)}_p
		&=
		\norm{D_2^\ell w^+(\tau,v;\Gamma(\tfrac{1}{n}))-D_2^\ell w^+(\tau,v;0)}_p\\
		&\leq
		\sup_{\theta\in\Theta}\norm{D_3D_2^\ell w^+(\tau,v;\theta)}_p\Gamma(\tfrac{1}{n})
		\stackrel{\eqref{thmifb6}}{\leq}
		N^+\norm{v}_p\Gamma(\tfrac{1}{n})
	\end{align*}
	for all $(\tau,v)\in\cV_j^+$ and $n\geq N$ implying the error estimate in $(a_2)$. 

	(b) This results parallel to (a) using \tref{thmifb}(b). In particular, $\cW_{j,-}^0:=\cW^-(0)$ and $\cW_{j,-}^n:=\cW^-(\Gamma(\tfrac{1}{n}))$ are graphs over the linear bundle $\cV_j^-$. Since the right-hand sides of \eqref{ide0} are completely continuous due to \pref{prophammer}, we obtain from \cite[p.~89, Prop.~6.5]{martin:76} that the derivatives in the variational equation \eqref{var0} are compact operators in $L(L_d^p(\Omega))$. Therefore, \cite[Cor.~4.13]{russ:16} implies that $\cV_j^-$ is finite-dimensional, which transfers to $\cW_{j,-}^n$ for $n=0$ and $n\geq N$. Moreover, due to the invariance $\cW_{j,-}^n(t+1)=\sF_t^n(\cW_{j,-}^n(t))\subseteq X_n^d$ for all $t\in\I'$ one has $\cW_{j,-}^n\subseteq\I\tm X_n^d$ for $n\geq N$ and similarly the smoothing property required in Hypothesis (i) implies $\cW_{j,-}^0\subseteq\I\tm X(\Omega)$. 

	(c) is a direct consequence of \tref{thmifb}(c). 
\end{proof}
\subsection{Pseudo-center bundles}
Complementing \tref{thmpifb}, we next establish that also the intersections 
$$
	\cV_{i,j}:=\cV_i^+\cap\cV_j^-\quad\text{for }i<j
$$
and thus the spectral bundles $\cV_j=\cV_{j-1,1}$ of \eqref{var0} persist under nonlinear perturbations, as well as spatial discretization as invariant bundles. 
\begin{theorem}[pseudo-center bundles under discretization]\label{thmcenter}
	Let Hypotheses (i--iii) hold with $\I=\Z$, $0<\gamma_j<\gamma_i$, $i,j\in\N_0$, be as in $(S_1)$ or $(S_2)$, choose $\alpha_k<\beta_k$ so that
	\begin{align*}
		\alpha_k&<\gamma_k<\beta_k,&
		(\alpha_k,\beta_k)\cap\Sigma_\Z&=\emptyset
		\quad\text{ for }k\in\set{i,j}
	\end{align*}
	and let $\cV_i^+,\cV_j^-$ denote the associated linear bundles of the variational equation \eqref{var0}. Then there are $K_i,K_j\geq 1$ depending only on \eqref{var0}, so that if $K_{\max}:=\!\displaystyle\max_{k\in\set{i,j}}(K_k^2+2K_k)$, 
	\begin{align}
		L&<\frac{\delta_{\max}}{2K_{\max}},&
		\delta_{\max}&:=\min_{k\in\set{i,j}}\frac{\beta_k-\alpha_k}{2}, 
		\label{thmcenter1}
	\end{align}
	$\delta\in(4K_{\max}L,\delta_{\max}]$ is fixed and 
	$\Gamma_k:=[\alpha_k+\delta,\beta_k-\delta]$ for $k\in\set{i,j}$, then there exists a $N\in\N$ such that for $n=0$ and $n\geq N$ the \emph{pseudo-center bundle}
	\begin{align*}
		\cW_{i,j}^n
		&:=
		\cW_{i,+}^n\cap\cW_{j,-}^n\\
		&=
		\set{(\tau,u)\in\Z\tm L_d^p(\Omega):
		\begin{array}{l}
		\sup_{\tau\leq t}\gamma_i^{\tau-t}\norm{\vphi^n(t;\tau,u)}_p<\infty\text{ and there}\\
		\text{exists a solution }(\phi_t)_{t\in\Z}\text{ of \eqref{iden} such}\\
		\text{that }\sup_{t\leq\tau}\gamma_j^{\tau-t}\norm{\phi_t}_p<\infty\text{ and }\phi_\tau=u
		\end{array}
		}
	\end{align*}
	is a finite-dimensional (unless for spectra $(S_1^1)$ and $j=J$) invariant bundle of \eqref{iden}, which is independent of $\gamma_i\in\Gamma_i$, $\gamma_j\in\Gamma_j$ and allows the representation
	$$
		\cW_{i,j}^n
		=
		\set{(\tau,v+w_{i,j}^n(\tau,v))\in\Z\tm L_d^p(\Omega):(\tau,v)\in\cV_{i,j}}
		\subseteq
		\begin{cases}
			\Z\tm X(\Omega),&n=0,\\
			\Z\tm X_n^d,&n\geq N
		\end{cases}
	$$
	with continuous mappings $w_{i,j}^n:\cV_{i,j}\to L_d^p(\Omega)$ satisfying $w_{i,j}^n(\tau,0)=0$ and 
	\begin{itemize}
		\item[(a)] there exist $C_0(n),N_0>0$ with $\lim_{n\to\infty}C_0(n)=\max_{k\in\set{i,j}}\frac{K_k^2L}{\delta-(K_k^2+2K_k)L}=:C_0(0)$ and for all $(\tau,v)\in\cV_{i,j}$ holds
		\begin{align}
			\lip w_{i,j}^n(\tau,\cdot)&\leq C_0(n),&
			\norm{w_{i,j}^n(\tau,v)-w_{i,j}^0(\tau,v)}_p
			&\leq
			N_0\Gamma(\tfrac{1}{n})\norm{v}_p,
			\label{thmcenter4}
		\end{align}

		\item[(b)] if $\alpha_i^m<\beta_i$, $\alpha_j<\beta_j^m$ and 
		$$
			\delta_{\max}
			:=
			\min\Bigl\{
			\tfrac{\beta_i-\alpha_i}{2},
			\tfrac{\beta_j-\alpha_j}{2},
			\alpha_i\bigl(\sqrt[m]{\tfrac{\alpha_i+\beta_i}{\alpha_i+\alpha_i^m}}-1\bigr),
			\beta_j\bigl(1-\sqrt[m]{\tfrac{\alpha_j+\beta_j}{\beta_j+\beta_j^m}}\bigr)
			\Bigr\},
		$$
		then the derivatives $D_2^\ell w_{i,j}^n:\cV_{i,j}\to L_\ell(L_d^p(\Omega))$ exist up to order $\ell\leq m$ as continuous maps and for $1\leq\ell<m$ there exist $N_\ell>0$ such that
		\begin{align*}
			\norm{D_2^\ell w_{i,j}^n(\tau,v)-D_2^\ell w_{i,j}^0(\tau,v)}_{L_\ell(L_d^p(\Omega))}
			&\leq
			N_\ell\Gamma(\tfrac{1}{n})\norm{v}_p
			\fall(\tau,v)\in\cV_{i,j}. 
		\end{align*}
	\end{itemize}
\end{theorem}
The magnitude of $N$ is determined by the same factors as in \tref{thmpifb}. We furthermore point out spectra $(S_1^1)$ and $j=J$, in which $\cW_{i,J}^n=\cW_{i,+}^n$ for all $i<J$. 
\begin{remark}[extended hierarchy]
	For dichotomy spectra $\Sigma_\Z$ of the form $(S_1)$ choose $J\in\N$ as in \sref{sec22}, while for spectra $(S_2)$ let $J\in\N$ be arbitrary. Then it results from their dynamical characterization and \eqref{noA2} that the bundles constructed in Thms.~\ref{thmpifb} and \ref{thmcenter} satisfy the following inclusions denoted as \emph{extended hierarchy}
	$$
	\begin{array}{ccccccccccc}
		\cW_{J,+}^n & \subset & \cW_{J-1,+}^n & \subset & \ldots & \subset & \cW_{2,+}^n & \subset & \cW_{1,+}^n & \subset & \Z\tm L_d^p(\Omega)\\
		&&\cup&&&&\cup&&\cup&&\cup\\
		&&\cW_{J-1,J}^n&\subset&\ldots&\subset&\cW_{2,J}^n&\subset&\cW_{1,J}^n&\subset&\cW_{J,-}^n\\
		&&&&&&&&\cup&&\cup\\
		&&&&&\ddots&&&\vdots&&\vdots\\
		&&&&&&&&\cup&& \cup\\
		&&&&&&&&\cW_{1,2}^n&\subset & \cW_{2,-}^n\\
		&&&&&&&&&&\cup\\
		&&&&&&&&&& \cW_{1,-}^n.
	\end{array}
	$$
\end{remark}

\begin{figure}[ht]
	\includegraphics[scale=0.5]{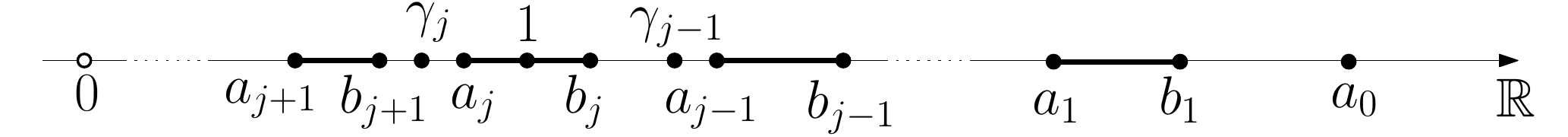}
	\caption{Nonhyperbolic situation where $1$ is contained in the spectral interval $[a_j,b_j]$}
	\label{fig5}
\end{figure}
\begin{remark}[persistence and center bundles]
	In a \emph{nonhyperbolic} situation $1\in[a_j,b_j]$ for some $1<j<J$ (see \fref{fig5}), 
	$\cW_s=\cW_{j,+}^0$ is the stable, 
	$\cW_{cs}:=\cW_{j-1,+}^0$ the \emph{center-stable}, 
	$\cW_{cu}:=\cW_{j,-}^0$ the \emph{center-unstable} and 
	$\cW_u=\cW_{j-1,-}^0$ the unstable bundle of \eqref{ide0}, while the \emph{center bundle} $\cW_c:=\cW_{j-1,j}^0$ completes the \emph{classical hierarchy} (cf.~\eqref{chier}). Now a center bundle, by definition, is graph over the spectral bundle $\cV_j$ associated to the spectral interval containing $1$. 
	Although the center bundle $\cW_c$ of \eqref{ide0} persists as invariant bundles $\cW_{j-1,j}^n$ for all $n\geq N$, in general these sets $\cW_{j-1,j}^n$ fail to be center bundles of the discretizations \eqref{iden}. This is due to the fact that $1$ needs not to be contained in the dichotomy spectrum $\Sigma_\Z^n$ of the discretized variational equations \eqref{varn}. For instance, one can think of a singleton interval $[a_j,b_j]=\set{1}$ moving away from $1$ under discretization or the fact that a solid spectral interval $[a_j,b_j]$ splits into subintervals, which is possible by the upper-semicontinuity of the spectrum $\Sigma_\Z$. 
\end{remark}

\begin{proof}[Proof of \tref{thmcenter}]
	Above all, choose $i<j$ fixed and we stay in the terminology established in the proof of \tref{thmpifb}. 

	(I) Our present assumptions yield that \tref{thmifb} applies to IDEs \eqref{eqtheta} in $X=L_d^p(\Omega)$. Therefore, for each $\theta\in\Theta$ there exists a $\gamma_i$-stable bundle $\cW^+(\theta)$ being a graph of a mapping $w^+(\cdot;\theta):\cV_i^+\to L_d^p(\Omega)$ and a $\gamma_j$-unstable bundle $\cW^-(\theta)$ represented as graph of $w^-(\cdot;\theta):\cV_j^-\to L_d^p(\Omega)$. The fact that their intersection can be represented as graph of a mapping $w(\cdot;\theta):\I\tm L_d^p(\Omega)\to L_d^p(\Omega)$ over $\cV:=\cV_{i,j}$ is shown as in the proof of \cite[p.~208, Prop.~4.2.17]{poetzsche:10}. Hence, it remains to establish the convergence statements. For this purpose, we choose $\theta\in\Theta$ and obtain from \tref{thmifb} that
	\begin{align*}
		\lip w^+(\tau,\cdot;\theta)&\stackrel{\eqref{noA3}}{\leq}\frac{2K_i^2L}{\delta-4K_iL}<1,&
		\lip w^-(\tau,\cdot;\theta)&\stackrel{\eqref{noA7}}{\leq}\frac{2K_j^2L}{\delta-4K_jL}<1
		\fall\tau\in\Z
	\end{align*}
	due to our strengthened assumptions \eqref{thmcenter1} (in comparison to \eqref{nosmall}); for this note that $\abs{\theta}\bar L\leq L$ due to $\theta\in\Theta$. This implies that
	\begin{align*}
		T_\tau:L_d^p(\Omega)^2\tm L_d^p(\Omega)\tm\Theta&\to L_d^p(\Omega)^2,&
		T_\tau(u,w,v,\theta)&:=
		\begin{pmatrix}
			w^+(\tau,v+w;\theta)\\
			w^-(\tau,u+v;\theta)
		\end{pmatrix}
	\end{align*}
	is a contraction in its first two arguments $(u,w)\in L_d^p(\Omega)^2$, uniformly in the parameters $(v,\theta)\in L_d^p(\Omega)\tm\Theta$, with the Lipschitz constant $\max_{k\in\set{i,j}}\frac{2K_k^2L}{\delta-4K_kL}<1$. Based on its unique fixed point $(\Upsilon_\tau^+,\Upsilon_\tau^-)(v,\theta)\in L_d^p(\Omega)^2$ we define the mapping
	\begin{equation}
		w(\tau,v;\theta):=\Upsilon_\tau^+(v,\theta)+\Upsilon_\tau^-(v,\theta)
		\fall\tau\in\I,\,v\in L_d^p(\Omega),\,\theta\in\Theta, 
		\label{nowest}
	\end{equation}
	depending only on the projection of $v\in L_d^p(\Omega)$ onto $\cV(\tau)$. Given $\theta,\bar\theta\in\Theta$ and with $\Upsilon_\tau:=(\Upsilon_\tau^+,\Upsilon_\tau^-)$, using the Lipschitz estimates \eqref{noA4}, \eqref{noA8} from \tref{thmifb} we obtain as in the proof of \cite[Thm.~4.2]{poetzsche:03} that there exists a constant $\tilde N_0\geq 0$ such that
	\begin{equation}
		\norm{\Upsilon_\tau(v,\theta)-\Upsilon_\tau(v,\bar\theta)}
		\leq
		\tilde N_0\norm{v}_p\abs{\theta-\bar\theta}
		\fall(\tau,v)\in\cV. 
		\label{no17s}
	\end{equation}
	Referring to \eqref{nowest} this guarantees 
	\begin{equation}
		\norm{w(\tau,v;\theta)-w(\tau,v;\bar\theta)}_p
		\leq
		2\tilde N_0\norm{v}_p\abs{\theta-\bar\theta}\fall(\tau,v)\in\cV,\,\theta,\bar\theta\in\Theta. 
		\label{susy}
	\end{equation}

	(II) Now suppose that both spectral gap conditions $\alpha_i^m<\beta_i$, $\alpha_j<\beta_j^m$ hold yielding that $w^+(\tau,\cdot)$ and $w^-(\tau,\cdot)$ are of class $C^m$ with globally bounded partial derivatives of order $1\leq\ell\leq m$ w.r.t.\ $v$ due to \eqref{noA5} and \eqref{noA9}. In order to establish that there exist constants $\tilde N_\ell\geq 0$ such that
	\begin{equation}
		\norm{D_2D_1^\ell\Upsilon_\tau(v,\theta)}_{L_\ell(L_d^p(\Omega),L_d^p(\Omega)^2)}
		\leq
		\tilde N_\ell\norm{v}_p\fall(\tau,v)\in\cV,\,\theta\in\Theta,
		\label{star2}
	\end{equation}
	we proceed by mathematical induction. For $\ell=0$ the bound on $D_2\Upsilon_\tau$ results from the Lipschitz condition \eqref{no17s} combined with \cite[p.~363, Prop.~C.1.1]{poetzsche:10}. Differentiating the fixed-point identity $\Upsilon_\tau(v,\theta)\equiv T_\tau(\Upsilon_\tau(v,\theta),v,\theta)$ on $\cV\tm\Theta$ w.r.t.\ $\theta$ first implies
	\begin{equation}
		D_2\Upsilon_\tau(v,\theta)
		\equiv
		D_{(1,2)}T_\tau(\Upsilon_\tau(v,\theta),v,\theta)D_2\Upsilon_\tau(v,\theta)+D_4T_\tau(\Upsilon_\tau(v,\theta),v,\theta)
		\label{susy2}
	\end{equation}
	on $\cV\tm\Theta$ and then thanks to the Product and Chain Rule (with partially unfolded derivative tree) the derivatives of \eqref{susy2} w.r.t.\ $v$ of order $1\leq\ell<m$ read as
	\begin{align*}
		D_2D_1^\ell\Upsilon_\tau(v,\theta)
		\equiv&
		\underbrace{D_{(1,2)}T_\tau(\Upsilon_\tau(v,\theta),v,\theta)}_{\norm{\cdot}<1}D_2D_1^\ell\Upsilon(v,\theta)\\
		&+
		D_v^\ell\intoo{
		D_{(1,2)}T_\tau(\Upsilon_\tau(v,\theta),v,\theta)D_2\Upsilon(v,\theta)
		+
		D_4T_\tau(\Upsilon_\tau(v,\theta),v,\theta)}
	\end{align*}
	on $\cV\tm\Theta$. Here, the norm bound on $D_{(1,2)}T_\tau(\Upsilon_\tau(v,\theta),v,\theta)$ results from the contraction property of $T_\tau(\cdot,v,\theta)$ and again \cite[p.~363, Prop.~C.1.1]{poetzsche:10}. Then the second term in the above sum is a sum of products. In each of them a derivative of $\Upsilon_\tau$ w.r.t.\ $\theta$ occurs exactly once, which yields a factor containing the term $\norm{v}_p$ (note \eqref{thmifb6} and \eqref{thmifb6s}), while the remaining factors are derivatives of $\Upsilon_\tau$ w.r.t.\ $v$, being bounded due to our induction hypothesis (see \eqref{noA5} and \eqref{noA9}). This establishes \eqref{star2} from which one has
	\begin{equation}
		\norm{D_3D_2^\ell w(\tau,v;\theta)}_{L_\ell(L_d^p(\Omega),L_d^p(\Omega))}
		\leq
		2\tilde N_\ell\norm{v}_p\fall(\tau,v)\in\cV,\,\theta\in\Theta.
		\label{no22}
	\end{equation}

	(III) After the prelude in steps (I) and (II) we finally define
	\begin{align*}
		\cW_{i,j}^0&:=\cW^+(0)\cap\cW^-(0),&
		w_{i,j}^0&:=w(\cdot,0)|_{\cV},\\
		\cW_{i,j}^n&:=\cW^+(\Gamma(\tfrac{1}{n}))\cap\cW^-(\Gamma(\tfrac{1}{n})),&
		w_{i,j}^n&:=w(\cdot,\Gamma(\tfrac{1}{n}))|_{\cV}\fall n\geq N
	\end{align*}
	and argue as above, i.e.\ by referring to \cite[p.~208, Prop.~4.2.17]{poetzsche:10} that it remains to establish the convergence estimates given in (a) and (b) for $n\geq N$: 
	
	(a) Then the error estimate in \eqref{thmcenter4} results from step (I) as for $(\tau,v)\in\cV_{i,j}$ one has
	$$
		\norm{w_{i,j}^n(\tau,v)-w_{i,j}^0(\tau,v)}_p
		=
		\norm{w(\tau,v;\Gamma(\tfrac{1}{n}))-w(\tau,v;0)}_p
		\stackrel{\eqref{susy}}{\leq}
		2\tilde N_0\Gamma(\tfrac{1}{n})\norm{v}_p.
	$$

	(b) Using the Mean Value Estimate we derive from step (II) for $1\leq\ell<m$ that
	\begin{align*}
		\norm{D_2^\ell w_{i,j}^n(\tau,v)-D_2^\ell w_{i,j}^0(\tau,v)}_p
		&=
		\norm{D_2^\ell w(\tau,v;\Gamma(\tfrac{1}{n}))-D_2^\ell w(\tau,v;0)}_p\\
		&\leq
		\sup_{\theta\in\Theta}\norm{D_3D_2^\ell w(\tau,v;\theta)}\Gamma(\tfrac{1}{n})
		\stackrel{\eqref{no22}}{\leq}
		2\tilde N_\ell \Gamma(\tfrac{1}{n})\norm{v}_p
	\end{align*}
	for all $(\tau,v)\in\cV_{i,j}$. This establishes claim (b). 
\end{proof}
\subsection{Applicability}
\label{seccutoff}
One cannot expect general IDEs 
\begin{align}
	\tag{$\tilde I_0$}
	u_{t+1}&=\tilde\sF_t(u_t),&
	\tilde\sF_t(u)&:=\int_\Omega k_t(\cdot,y)\tilde g_t(y,u(y))\d y, 
	\label{ide2}
\end{align}
arising in real-world applications \cite{jacobsen:jin:lewis:15,lutscher:19} to satisfy global assumption such as \eqref{no7s}, \eqref{no62}, the global smallness condition \eqref{nosmall}, or that the trivial solution lies in the center of interest. Nonetheless, given any solution $(\phi_t)_{t\in\I}$ of \eqref{ide2} the \emph{equation of perturbed motion}
$$
	u_{t+1}
	=
	\tilde\sF_t(u_t+\phi_t)-\tilde\sF_t(\phi_t)
	=
	\int_\Omega k_t(\cdot,y)\bar g_t(y,u_t(y))\d y
$$
with growth function $\bar g_t(x,z):=\tilde g_t(x,z+\phi_t(x))-\tilde g_t(x,\phi_t(x))$ has the trivial solution. The dynamics in its vicinity is the same as that of \eqref{ide2} near the reference solution $\phi$. 

The solutions of \eqref{ide2} being relevant in applications have (essentially) bounded values $\phi_t$. Under this premise a suitable modification of the growth function $\bar g_t:\Omega\tm\K^d\to\K^n$ allows to apply \tref{thmpifb} and \ref{thmcenter}. For this purpose assume that $\bar g_t$ satisfies Hypothesis~(ii) with the global conditions \eqref{no7s} and \eqref{no62} weakened to
\begin{itemize}
	\item $\lim_{z\to 0}\esup_{x\in\Omega}\abs{D_2\tilde g_t(x,z+\phi_t(x))-D_2\tilde g_t(x,\phi_t(x))}=0$ for all $t\in\I'$, 

	\item for each $r>0$ and $t\in\I'$ there exists a $\Lambda_t(r)\geq 0$ such that
	$$
		\abs{\tilde g(x,z)-\tilde g(x,\bar z)}
		\leq
		\Lambda_t(r)\abs{z-\bar z}
		\fall z,\bar z\in\bar B_r(\phi_t(x))\text{ and $\mu$-a.a.\ }x\in\Omega.
	$$
\end{itemize}
Given $\rho>0$, let $\chi_\rho:\K^d\to[0,1]$ denote a \emph{cut-off function}, i.e.\ a $C^\infty$-function constant $\chi_\rho(z)\equiv 1$ on $\bar B_\rho(0)$, $\chi_\rho(z)\in(0,1)$ for $z\in B_{2\rho}(0)\setminus\bar B_\rho(0)$ and vanishing $\chi_\rho(z)\equiv 0$ on the remaining set $\K^d\setminus B_{2\rho}(0)$ (cf.\ \cite[p.~369, Prop.~C.2.16]{poetzsche:10}). First, thanks to this construction the growth functions $\bar g_t$ and
\begin{equation}
	g_t(x,z):=\bar g_t(x,\chi_\rho(z)z)=\tilde g_t(x,\chi_\rho(z)z+\phi_t(x))-\tilde g_t(x,\phi_t(x))
	\label{gmod}
\end{equation}
coincide for all $z\in\bar B_\rho(0)$ and a.a.\ $x\in\Omega$. Second, by \cite[p.~370, Prop.~C.2.17]{poetzsche:10} the radius $\rho>0$ can be chosen so small that under the balancing conditions \eqref{no7s}, \eqref{no62} the resulting nonlinearities $\sN_t,\bar\sN_t:L_d^p(\Omega)\to L_d^p(\Omega)$ in the proof of \tref{thmpifb} satisfy the assumptions of \tref{thmifb}. This allows to obtain the particular fiber bundles $\cW_{i,+}^n,\cW_{j,-}^n$ and $\cW_{i,j}^n$ of the IDE \eqref{ide0}, whose growth function $g_t$ is given in \eqref{gmod}. 

In order to transfer this information back to the original IDE \eqref{ide2} we note that the dynamics of \eqref{ide2} in the vicinity $\{(t,u)\in\I\tm L_d^p(\Omega):\,\norm{u-\phi_t}_\infty\leq\rho\}$ of the reference solution $\phi$ and the behavior of the modified problem \eqref{ide0} in the nonautonomous set $\{(t,u)\in\I\tm L_d^p(\Omega):\,\norm{u}_\infty\leq\rho\}$ are related by a translation via the solution $\phi$. 

\section{Galerkin discretizations of integrodifference equations}
\label{sec4}
This section demonstrates that the convergence assumption from Hypothesis (iii) holds for various Galerkin approximations. They are a special case of general projection methods (cf.~\cite[Sect.~2]{atkinson:92} or \cite[pp.~448ff]{atkinson:han:00}) to discretize IDEs spatially. Here, let $a<b$. 
\subsection{Galerkin methods on $\Omega=(a,b)$}
On the interval $\Omega:=(a,b)$, choose $n\in\N_0$, the nodes
\begin{align}
	a&=x_0^n<x_1^n<\ldots<x_{m_n}^n=b
	\label{nonodes}
\end{align}
and set $m_0:=1$ (i.e.\ one obtains the entire interval $(a,b)$ for $n=0$), 
\begin{align*}
	h_n&:=\max_{i=1}^{m_n}(x_i^n-x_{i-1}^n),&
	I_i^n&:=[x_{i-1}^n,x_i^n]\fall 1\leq i\leq m_n. 
\end{align*}
We assume throughout that there exists a $C>0$ such that
\begin{align}
	h_n\leq\tfrac{C}{n}\fall n\in\N.
	\label{star}
\end{align}
\paragraph*{Piecewise constant approximation}
Let $\chi_{I_i^n}:\R\to\set{0,1}$, $1\leq i\leq m_n$, be the characteristic functions over the partition \eqref{nonodes}. With $X_n:=\spann\bigl\{\chi_{I_1^n},\ldots,\chi_{I_{m_n}^n}\bigr\}$ one defines the projections (see \cite[p.~305, 9.21]{alt:16})
$$
	\Pi_nu
	:=
	\sum_{i=1}^{m_n}\chi_{I_i^n}(\cdot)
	\intoo{\frac{1}{x_i^n-x_{i-1}^n}\int_{x_{i-1}^n}^{x_i^n}u(x)\d x}
	\fall u\in L_d^p(a,b)
$$
onto the piecewise constant functions $X_n^d$ and has the error estimate
$$
	\norm{[\id_{L_d^p(a,b)}-\Pi_n]u}_p
	\leq
	h_n\norm{u'}_p
	\fall n\in\N,\,u\in W_d^{1,p}(a,b). 
$$
This yields a discretization error \eqref{thmpifb4} with the space of weakly differentiable functions $X(a,b)=W_d^{1,p}(a,b)$ having first order derivatives in $L_d^p(a,b)$ and $\Gamma(\varrho)=C\varrho$. 
\paragraph*{Continuous piecewise linear approximation}
Let $\chi_i^n:\R\to[0,1]$, $0\leq i\leq m_n$, denote the hat functions over the partition \eqref{nonodes}. For $X_n:=\spann\bigl\{\chi_0^n,\ldots,\chi_{m_n}^n\bigr\}$ and continuous functions $u:(a,b)\to\K^d$ one defines the projections 
$$
	\Pi_nu(x)
	:=
	\sum_{i=1}^{m_n}\chi_{I_i^n}(x)
	\intoo{\frac{x_i^n-x}{x_{i^n}^n-x_{i-1}^n}u(x_{i-1}^n)
	+
	\frac{x-x_{i-1}^n}{x_{i^n}^n-x_{i-1}^n}u(x_i^n)}
	\fall u\in C_d(a,b)
$$
onto $X_n^d$. First, for each $l\in\set{1,2}$ one has the error estimate (cf.\ \cite[p.~309, 9.22]{alt:16})
$$
	\norm{[\id_{L_d^p(a,b)}-\Pi_n]u}_p
	\leq
	2h_n^l\norm{u^{(l)}}_p
	\fall n\in\N,\,u\in W_d^{l,p}(a,b). 
$$
Hence, \eqref{thmpifb4} holds with $X(a,b)=W_d^{l,p}(a,b)$ being the space of functions whose weak derivatives exist up to order $l\leq 2$ in $L_d^p(a,b)$ and $\Gamma(\varrho)=2C^l\varrho^l$. 

Second, for $\alpha\in(0,1]$ and $\alpha$-H\"older functions $u:[a,b]\to\K^d$ one establishes the error estimate (s.\ \cite[p.~457, (12.2.4)]{atkinson:han:00})
$$
	\norm{[\id_{L_d^p(a,b)}-\Pi_n]u}_p
	\leq
	(b-a)^{1/p}d^{1/p}h_n^\alpha\norm{u}_{C_d^\alpha[a,b]}
	\fall n\in\N,\,u\in C_d^\alpha[a,b]. 
$$
The condition \eqref{thmpifb4} is fulfilled with the $\K^d$-valued functions $X(a,b)=C_d^\alpha[a,b]$ satisfying a H\"older condition with exponent $\alpha$ and $\Gamma(\varrho)=(b-a)^{1/p}d^{1/p}C^\alpha\varrho^\alpha$. 
\subsection{Galerkin methods on polygonal domains $\Omega\subset\R^\kappa$ with $\kappa=2,3$}
Let $\kappa=2,3$ and $\Omega\subset\R^\kappa$ be a polygonal domain (open, bounded, connected). For a sequence $(h_n)_{n\in\N}$ in $(0,\infty)$ satisfying \eqref{star} suppose the representations
$$
	\Omega=\bigcup_{K\in\fT_n}K\fall n\in\N
$$
with a sequence of regular triangulations $\fT_n$, i.e.\ there exists a constant $T>0$ such that each family $\fT_n$ consists of finitely many polyhedra $K\subseteq\Omega$ and
\begin{align*}
	\max_{K\in\fT_n}\diam K&\leq h_n,&
	\max_{K\in\fT_n}\frac{\diam K}{\sup\set{\diam S:\,S\text{ is a ball contained in }K}}\leq T
\end{align*}
for $n\in\N$. Let $P_k(K)$ be the space of polynomials of (maximal) degree $k\in\N$ in variables $x_1,\ldots,x_\kappa\in\R$ over $K$ and $X_n:=\{v\in C(\Omega):\,v|_K\in P_k(K)\text{ for all }K\in\fT_n\}$. For $1<l\leq k+1$ it results from \cite[p.~96, (3.5.9)]{quarteroni:valli:94} that there exists a $\bar C>0$ with
$$
	\norm{[\id_{L_d^2(\Omega)}-\Pi_n]u}_2
	\leq
	\bar Ch_n^l
	\sqrt{\sum_{\abs{\alpha}=l}\norm{D^\alpha u}_2^2}
	\fall n\in\N,\,u\in H_d^l(\Omega). 
$$
Consequently the error estimate \eqref{thmpifb4} holds for functions $X(\Omega)=H_d^l(\Omega)$ being weakly differentiable up to order $l\leq k+1$ with derivatives in $L_d^2(\Omega)$ and $\Gamma(\varrho)=\bar CC^l\varrho^l$.
\subsection{Spectral Galerkin methods}
For a Hammerstein IDE \eqref{ide0} let us suppose that the kernels and growth functions are of the form $k_t=k$ and $g_t(x,z):=\alpha_tg(x,z)$ for all $t\in\I'$ with a sequence $\alpha_t\in\R$ and functions $k:\Omega\tm\Omega\to\K^{d\tm n}$, $g:\Omega\tm\K^d\to\K^n$ such that Hypotheses (i--iii) are satisfied with $p=2$. The linear operator $\sL\in L(L_d^2(\Omega))$, $\sL u:=\int_\Omega k(\cdot,y)D_2g(y,0)u(y)\d y$ is compact and in case
$$
	k(x,y)D_2g(y,0)=[k(y,x)D_2g(x,0)]^\ast\quad\text{for a.a.\ }x,y\in\Omega
$$
also self-adjoint. Consequently there exists a complete orthonormal set consisting of eigenfunctions $\chi_i:\Omega\to\K^d$ for $\sL$. Given this, we define the orthogonal projections
$$
	\Pi_nu:=\sum_{i=1}^n\iprod{u,\chi_i}\chi_i\fall n\in\N
$$
onto $X_n^d:=\spann\set{\chi_1,\ldots,\chi_n}$ and it is not difficult to arrive at
$$
	\norm{[\id_{L_d^2(\Omega)}-\Pi_n]u}_2
	\leq
	\sqrt{\sum_{j=n+1}^\infty\abs{\iprod{u,\chi_j}}^2}\fall n\in\N.
$$
Whence, the convergence function $\Gamma$ in the error estimate \eqref{thmpifb4} depends on the decay of the Fourier coefficients $\iprod{u,\chi_j}$. If the kernel comes from a Sturm-Liouville problem, then related results are treated in \cite[pp.~275, Sect.~5.2]{canuto:etal:06}. 
\section{Perspectives}
An alternative to the $L^p$-setting we elaborated on in this paper, is to employ the strictly decreasing scale of H\"older spaces $C^\alpha(\Omega,\K^d)$ over a compact habitat $\Omega\subset\R^\kappa$. For this purpose and exponents $0<\beta<\alpha\leq 1$ one considers 
\begin{itemize}
	\item Fredholm operators $\sK_t\in L(C^\beta(\Omega,\K^n),C^\alpha(\Omega,\K^d))$ as defined in \eqref{defK}, 

	\item Nemytskii operators $\sG_t:C^\alpha(\Omega,\K^d)\to C^\beta(\Omega,\K^n)$ as given in \eqref{defG}.
\end{itemize}
The reason for working with different H\"older exponents are the pathological mapping properties of the Nemytskii operators $\sG_t$ in case $\alpha=\beta$. Indeed, then any globally Lipschitz $\sG_t$ is already affine-linear, cf.~\cite{matkowski:09}. A particular feature of such a H\"older setting is that persistence and convergence of invariant fibers can be shown for Nystr\"om discretizations of \eqref{ide0} (see \cite{atkinson:92,poetzsche:20a}). 
\appendix
%
%
\section{Perturbation of invariant bundles}
\label{appA}
In this appendix, $\I$ is an unbounded discrete interval, while $(X,\norm{\cdot})$ denotes a Banach space. We consider non\-auto\-nomous semilinear difference equations
\begin{equation}
	\tag{$\Delta_\theta$}
	u_{t+1}=\sL_tu_t+\sN_t(u_t)+\theta\bar\sN_t(u_t)
	\label{fbdeq}
\end{equation}
in $X$ depending on a parameter $\theta\in\R$ and having the general solution $\vphi(\cdot;\theta)$. Suppose that the following assumptions hold: 
\begin{itemize}
	\item[\textbf{(H1)}] $\sL_t\in L(X)$, $t\in\I'$, there exists a projection-valued sequence $(P_t)_{t\in\I}$ in $L(X)$ and reals $K\geq 1$, $0<\alpha<\beta$ so that $P_{t+1}\sL_t=\sL_tP_t$ and $\sL_t|_{N(P_t)}:N(P_t)\to N(P_{t+1})$ is an isomorphism for all $t\in\I'$, as well as
	\begin{align}
		\norm{\Phi(t,s)P_s}&\leq K\alpha^{t-s},&
		\norm{\Phi(s,t)[\id_X-P_t]}&\leq K\beta^{s-t}\fall s\leq t. 
		\label{noh1}
	\end{align}

	\item[\textbf{(H2)}] The identities $\sN_t(0)\equiv \bar\sN_t(0)\equiv 0$ on $\I'$ hold for mappings $\sN_t,\bar\sN_t:X\to X$, $t\in\I'$, satisfying the Lipschitz estimates
	\begin{align*}
		L&:=\sup_{t\in\I'}\lip \sN_t<\infty,&
		\bar L&:=\sup_{t\in\I'}\lip\bar\sN_t<\infty.
	\end{align*}

	\item[\textbf{(H3)}] Let $m\in\N$ and $\sN_t,\bar\sN_t$ be $m$-times continuously differentiable such that
	\begin{align*}
		\hspace*{-10mm}
		\sup_{(t,u)\in\I'\tm X}\norm{D^\ell\sN_t(u)}_{L_\ell(X)}&<\infty,&
		\sup_{(t,u)\in\I'\tm X}\norm{D^\ell\bar\sN_t(u)}_{L_\ell(X)}&<\infty
		\fall 1<\ell\leq m.
	\end{align*}
\end{itemize}

For our central perturbation result, given $\tau\in\I$ and $\gamma>0$ the linear spaces
\begin{align*}
	\ell_{\tau,\gamma}^+&:=\set{(\phi_t)_{\tau\leq t}:\,\phi_t\in X\text{ and }\sup_{\tau\leq t}\gamma^{\tau-t}\norm{\phi_t}<\infty},\\
	\ell_{\tau,\gamma}^-&:=\set{(\phi_t)_{t\leq\tau}:\,\phi_t\in X\text{ and }\sup_{t\leq\tau}\gamma^{\tau-t}\norm{\phi_t}<\infty}
\end{align*}
and $\ell_\gamma:=\set{(\phi_t)_{t\in\Z}:\,\phi_t\in X\text{ and }\sup_{t\in\Z}\gamma^{-t}\norm{\phi_t}<\infty}$ are due and satisfy
\begin{align}
	\ell_{\tau,\gamma}^+&\subseteq\ell_{\tau,\bar\gamma}^+,&
	\ell_{\tau,\bar\gamma}^-&\subseteq\ell_{\tau,\gamma}^-\fall 0<\gamma\leq\bar\gamma.
	\label{noA2}
\end{align}

These preparations allows us to establish existence, smoothness and $\theta$-dependence of invariant bundles for \eqref{fbdeq} carrying most of the technical effort for the results above. 
\begin{theorem}[existence and perturbation of invariant bundles]\label{thmifb}
	Under the assumptions (H1--H2) with
	\begin{align*}
		L&<\frac{\delta_{\max}}{4K},&
		\delta_{\max}&:=\frac{\beta-\alpha}{2}
	\end{align*}
	we choose a fixed $\delta\in(4KL,\tfrac{\beta-\alpha}{2}]$ and set
	$\Gamma:=[\alpha+\delta,\beta-\delta]$, $\Theta:=\set{\theta\in\R:\,\bar L\abs{\theta}\leq L}$. If $\gamma\in\Gamma$, then the following statements are true for all $\theta\in\Theta$: 
	\begin{enumerate}
		\item[(a)] In case $\I$ is unbounded above, then the \emph{$\gamma$-stable bundle}
		$$
			\cW^+(\theta)
			:=
			\set{(\tau,u)\in\I\tm X:\,\vphi(\cdot;\tau,u;\theta)\in\ell_{\tau,\gamma}^+}
		$$
		is a forward invariant bundle of \eqref{fbdeq} independent of $\gamma\in\Gamma$ having the representation
		$
			\cW^+(\theta)
			=
			\set{(\tau,v+w^+(\tau,v;\theta))\in\I\tm X:\,v\in R(P_\tau)}
		$
		with a continuous mapping $w^+:\I\tm X\tm\Theta\to X$ satisfying for all $\tau\in\I$, $u\in X$ that
		\begin{itemize}
			\item[$(a_1)$] $w^+(\tau,0;\theta)=0$ and $w^+(\tau,u;\theta)=w^+(\tau,P_\tau u;\theta)\in N(P_\tau)$, 

			\item[$(a_2)$] $w^+:\I\tm X\tm\Theta\to X$ fulfills the Lipschitz estimates
			\begin{align}
				\lip w^+(\tau,\cdot;\theta)&\leq\frac{K^2(L+\abs{\theta}\bar L)}{\delta-2K(L+\abs{\theta}\bar L)},
				\label{noA3}\\
				\lip w^+(\tau,u;\cdot)&\leq\frac{2\delta K^3\bar L}{(\delta-4K\bar L)^2}\norm{u},
				\label{noA4}
			\end{align}

			\item[$(a_3)$] if also $(H3)$ holds, $\alpha^{m_+}<\beta$ and $\delta_{\max}:=\min\Bigl\{\tfrac{\beta-\alpha}{2},\alpha\bigl(\sqrt[{m_+}]{\frac{\alpha+\beta}{\alpha+\alpha^{m_+}}}-1\bigr)\Bigr\}$, then the derivatives $D_{(2,3)}^\ell w^+:\I\tm X\tm\Theta\to L_\ell(X\tm\R,X)$ exist up to order $\ell\leq m_+\leq m$ as continuous functions and there are $M^+,N^+>0$ such that
			\begin{align}
				\norm{D_2^\ell w^+(\tau,u;\theta)}_{L_\ell(X)}
				&\leq
				M^+\fall 1\leq\ell\leq m_+,
				\label{noA5}\\
				\norm{D_3D_2^\ell w^+(\tau,u;\theta)}_{L_\ell(X)}
				&\leq
				N^+\norm{u}\fall 0\leq l<m_+.
				\label{thmifb6}
			\end{align}
		\end{itemize}

		\item[(b)] In case $\I$ is unbounded below, then the \emph{$\gamma$-unstable bundle}
		$$
			\cW^-(\theta)
			:=
			\set{(\tau,u)\in\I\tm X:
			\begin{array}{l}
			\text{there exists a solution }(\phi_t)_{t\leq\tau}\in\ell_{\tau,\gamma}^-\\
			\text{of \eqref{fbdeq} satisfying }\phi_\tau=u
			\end{array}
			}
		$$
		is an invariant bundle of \eqref{fbdeq} independent of $\gamma\in\Gamma$ having the representation
		$
			\cW^-(\theta)
			=
			\set{(\tau,v+w^-(\tau,v;\theta))\in\I\tm X:\,v\in N(P_\tau)}
		$
		with a continuous mapping $w^-:\I\tm X\tm\Theta\to X$ satisfying for all $\tau\in\I$, $u\in X$ that
		\begin{itemize}
			\item[$(b_1)$] $w^-(\tau,0;\theta)=0$ and $w^-(\tau,u;\theta)=w^-(\tau,[\id_X-P_\tau]u;\theta)\in R(P_\tau)$, 

			\item[$(b_2)$] $w^-:\I\tm X\tm\Theta\to X$ fulfills the Lipschitz estimates
			\begin{align}
				\lip w^-(\tau,\cdot;\theta)&\leq\frac{K^2(L+\abs{\theta}\bar L)}{\delta-2K(L+\abs{\theta}\bar L)},
				\label{noA7}\\
				\lip w^-(\tau,u;\cdot)&\leq\frac{2\delta K^3\bar L}{(\delta-4K\bar L)^2}\norm{u}, 
				\label{noA8}
			\end{align}

			\item[$(b_3)$] if also $(H3)$ holds, $\alpha<\beta^{m_-}$ and $\delta_{\max}:=\min\Bigl\{\tfrac{\beta-\alpha}{2},\beta\bigl(1-\sqrt[{m_-}]{\frac{\alpha+\beta}{\beta+\beta^{m_-}}}\bigr)\Bigr\}$, then the derivatives $D_{(2,3)}^\ell w^-:\I\tm X\tm\Theta\to L_\ell(X\tm\R,X)$ exist up to order $\ell\leq m_-\leq m$ as continuous functions and there are $M^-,N^->0$ such that
			\begin{align}
				\norm{D_2^\ell w^-(\tau,u,\theta)}_{L_\ell(X)}
				&\leq
				M^-\fall 1\leq\ell\leq m_-,
				\label{noA9}\\
				\norm{D_3D_2^\ell w^-(\tau,u,\theta)}_{L_\ell(X)}
				&\leq
				N^-\norm{u}\fall 0\leq l<m_-.
				\label{thmifb6s}
			\end{align}
		\end{itemize}

		\item[(c)] If $\I=\Z$ and $L<\frac{\delta}{6K}$, then $\cW^+(\theta)\cap\cW^-(\theta)=\Z\tm\set{0}$ for $\theta\in\Theta$ and the zero solution is the only solution to \eqref{fbdeq} in $\ell_\gamma$ for any growth rate $\gamma\in\Gamma$. 
	\end{enumerate}
\end{theorem}
\begin{proof}
	Apply \cite[Thm.~3.3]{poetzsche:03} on the measure chains ${\mathbb T}=\I\subseteq\Z$. 
\end{proof}
\section{Fredholm and Nemytskii operators on $L^p$-spaces}
\label{appB}
Let $(\Omega,\fA,\mu)$ be a measure space with $\sigma$-algebra $\fA$ and a finite measure $\mu$, while $(Z,\abs{\cdot})$ denotes a finite-dimensional Banach space over $\K$. For $p\in[1,\infty)$ we define the space
\begin{align*}
	L^p(\Omega,Z)
	&:=
	\biggl\{u:\Omega\to Z
	\biggl|u\text{ is $\mu$-measurable with }\int_\Omega\abs{u}^p\d\mu<\infty
	\biggr\}
\end{align*}
of $Z$-valued $p$-integrable functions and equip it with $\norm{u}_p:=\intoo{\int_\Omega\abs{u}^p\d\mu}^{1/p}$ as norm.

Our analysis requires preparations on Fredholm integral operators
$$
	\sK v:=\int_\Omega k(\cdot,y)v(y)\d\mu(y). 
$$

First, for the kernel $k:\Omega\tm\Omega\to\K^{d\tm n}$ we suppose \emph{Hille-Tamarkin conditions}: There exist $p,q\in(1,\infty)$ such that
\begin{enumerate}
	\item[\textbf{(h1)}]\quad $k$ is $\mu\otimes\mu$-measurable, 

	\item[\textbf{(h2)}]\quad if $q'>1$ with $\tfrac{1}{q}+\tfrac{1}{q'}=1$, then 
	$
		\int_\Omega\biggl(\int_\Omega\abs{k(x,y)}^{q'}\d\mu(y)\biggr)^{p/q'}\d\mu(x)
		<
		\infty. 
	$
\end{enumerate}
\begin{proposition}\label{propmak}
	If $(h1$--$h2)$ hold with $p,q>1$, then $\sK\in L(L^q(\Omega,\K^n),L^p(\Omega,\K^d))$ is well-defined and compact with
	\begin{equation}
		\norm{\sK}_{L(L^q(\Omega,\K^n),L^p(\Omega,\K^d))}
		\leq
		\biggl(\int_\Omega\biggl(\int_\Omega\abs{k(x,y)}^{q'}\d\mu(y)\biggr)^{\tfrac{p}{q'}}\d\mu(x)\biggr)^{\tfrac{1}{p}}.
		\label{propmak1}
	\end{equation}
\end{proposition}
\begin{proof}
	See \cite[pp.~149--150, 5.12]{alt:16} and \cite[p.~340, 10.15]{alt:16} for compactness. 
\end{proof}

Second, let us furthermore consider the Nemytskii operator
$$
	[\sG(u)](x):=g(x,u(x))\quad\text{for $\mu$-a.a.\ } x\in\Omega
$$
induced by a growth function $g:\Omega\tm\K^d\to\K^n$. Given $m\in\N_0$ we assume that for all $0\leq\ell\leq m$ one has the following \emph{Carath{\'e}odory conditions}: 
\begin{enumerate}
	\item[\textbf{(c1)}]\quad $D_2^\ell g(x,\cdot):\K^d\to L_\ell(\K^d,\K^n)$ exists and is continuous for $\mu$-a.a.\ $x\in\Omega$, 

	\item[\textbf{(c2)}]\quad $D_2^\ell g(\cdot,z):\Omega\to L_\ell(\K^d,\K^n)$ is $\mu$-measurable on $\Omega$ for all $z\in\K^d$.
\end{enumerate}
\begin{proposition}\label{lemmag}
	Let $p,q\geq 1$, $c\in L^q(\Omega)$, $c_0\geq 0$ and suppose $(c1$--$c2)$ hold with $m=0$. If the growth estimate
	$$
		\abs{g(x,z)}
		\leq
		c(x)+c_0\abs{z}^{\tfrac{p}{q}}
		\quad\text{for $\mu$-a.a.~$x\in\Omega$ and all }z\in\K^d
	$$
	is satisfied, then $\sG:L^p(\Omega,\K^d)\to L^q(\Omega,\K^n)$ is well-defined, bounded and continuous. 
\end{proposition}
\begin{proof}
	See \cite[p.~63, Thm.~5.1]{precup:02}.
\end{proof}

\begin{corollary}\label{cormag}
	Let $1\leq q<p$. If there exists a $\mu$-measurable function $\lambda:\Omega\to\R_+$ satisfying
	$$
		\abs{g(x,z)-g(x,\bar z)}
		\leq
		\lambda(x)\abs{z-\bar z}\quad\text{for $\mu$-a.a.\ }x\in\Omega\text{ and all }z,\bar z\in\K^d
	$$
	and $\int_\Omega\lambda(y)^{\tfrac{p}{p-q}}\d\mu(y)<\infty$, then $\sG:L^q(\Omega,\K^d)\to L^q(\Omega,\K^n)$ is globally Lipschitz with
	\begin{equation}
		\lip\sG
		\leq
		\biggl(\int_\Omega\lambda(y)^{\tfrac{p}{p-q}}\d\mu(y)\biggr)^{\tfrac{p-q}{p}}. 
		\label{cormag2}
	\end{equation}
\end{corollary}
\begin{proof}
	Given $u,\bar u\in L^p(\Omega,\K^d)$, we obtain
	$$
		\norm{\sG(u)-\sG(\bar u)}_q^q
		=
		\int_\Omega\abs{g(y,u(y))-g(y,\bar u(y))}^q\d\mu(y)
		\leq
		\int_\Omega\lambda(y)^q\abs{u(y)-\bar u(y)}^q\d\mu(y)
	$$
	and the H\"older inequality $|\int_\Omega v_1(y)v_2(y)\d\mu(y)|\leq\norm{v_1}_{r'}\norm{v_2}_r$ for $\tfrac{1}{r}+\tfrac{1}{r'}=1$ leads to
	\begin{align*}
		\norm{\sG(u)-\sG(\bar u)}_q
		&\leq
		\norm{\lambda^q}_{r'}^{1/q}\norm{\abs{u(\cdot)-\bar u(\cdot)}^q}_r^{1/q}\\
		&=
		\biggl(\int_\Omega\lambda(y)^{qr'}\d\mu(y)\biggr)^{\tfrac{1}{qr'}}
		\biggl(\int_\Omega\abs{u(y)-\bar u(y)}^{qr}\d\mu(y)\biggr)^{\tfrac{1}{qr}}.
	\end{align*}
	We now choose $r>0$ such that $p=rq$ (note that $q<p$ guarantees $1<r$) and obtain $qr'=\tfrac{p}{p-q}$, as well as $r'=\tfrac{r}{r-1}$, thus
	$$
		\norm{\sG(u)-\sG(\bar u)}_q
		\leq
		\biggl(\int_\Omega\lambda(y)^{\tfrac{p}{p-q}}\d\mu(y)\biggr)^{\tfrac{p-q}{p}}\norm{u-\bar u}_p,
	$$
	which guarantees the Lipschitz condition \eqref{cormag2}. 
\end{proof}

\begin{proposition}\label{lemdiff}
	Let $m\in\N$, $p,q\geq 1$ with $mq<p$, $c\in L^{\tfrac{pq}{p-mq}}(\Omega)$, $c_0,\ldots,c_m\geq 0$ and suppose $(c1$--$c2)$ hold. If for $0\leq\ell\leq m$ the growth conditions 
	\begin{equation}
		\abs{D_2^\ell g(x,z)}_{L_\ell(\K^d,\K^n)}
		\leq
		c(x)+c_\ell\abs{z}^{\tfrac{p-\ell q}{q}}
		\quad\text{for $\mu$-a.a.\ }x\in\Omega\text{ and all }z\in\K^d
		\label{growthest}
	\end{equation}
	are satisfied, then $\sG:L^p(\Omega,\K^d)\to L^q(\Omega,\K^n)$ is of class $C^m$ and for each $0\leq\ell\leq m$, $u\in L^p(\Omega,\K^d)$ the following holds: 
	\begin{enumerate}
		\item[(a)] $[D^\ell\sG(u)v_1\cdots v_\ell](x)=D_2^\ell g(x,u(x))v_1(x)\cdots v_\ell(x)$ for $\mu$-a.a.\ $x\in\Omega$ and functions $v_1,\ldots,v_\ell\in L^q(\Omega,\K^d)$, 

		\item[(b)] $\norm{D^\ell\sG(u)}_{L_\ell(L^p(\Omega,\K^d),L^q(\Omega,\K^n))}\leq\norm{c}_{\tfrac{pq}{p-\ell q}}+c_\ell\norm{u}_p^{\tfrac{p-\ell q}{q}}$. 
	\end{enumerate}
\end{proposition}
Due to the lack of a suitable reference we provide an explicit proof. 
\begin{proof}
	Throughout, let $u\in L^p(\Omega,\K^d)$ and formally define for $\mu$-a.a.\ $x\in\Omega$ that
	\begin{align*}
		\sG^\ell(u)(x)&:=D_2^\ell g(x,u(x))\in L_\ell(\K^d,\K^n)\fall 0\leq\ell\leq m,\\
		[\bar\sG^\ell(u)h](x)&:=D_2^\ell g(x,u(x))h(x)\in L_{\ell-1}(\K^d,\K^n)\fall 0<\ell\leq m
	\end{align*}
	and $h\in L^p(\Omega,\K^d)$. 
	First, being $p$-integrable, $u$ is $\mu$-measurable on $\Omega$. Since $D_2^\ell g$ satisfy Carath{\'e}odory conditions, because of \cite[p.~62, Lemma~5.1]{precup:02} also $x\mapsto D_2^\ell g(x,u(x))$ is $\mu$-measurable on $\Omega$. Second, if we set $q_\ell:=\tfrac{pq}{p-\ell q}\geq 1$, then \eqref{growthest} become 
	$$
		\abs{D_2^\ell g(x,z)}
		\leq
		c(x)+c_\ell\abs{z}^{p/q_\ell}
		\fall 0\leq\ell\leq m,\text{$\mu$-a.a.~$x\in\Omega$ and all }z\in\K^d. 
	$$

	(I) Claim: $\sG^\ell:L^p(\Omega,\K^d)\to L^{q_\ell}(\Omega,L_\ell(\K^d,\K^n))$ is well-defined and bounded for all $0\leq\ell\leq m$.\\
Using the growth condition \eqref{growthest} and the $L^{q_\ell}$-triangle inequality we obtain
	\begin{align*}
		\intoo{\int_\Omega\abs{D_2^{\ell}g(y,u(y))}^{q_\ell}\d\mu(y)}^{1/q_\ell}
		&\leq
		\intoo{\int_\Omega\abs{c(y)+c_\ell\abs{u(y)}^{p/q_\ell}}^{q_\ell}\d\mu(y)}^{1/q_\ell}\\
		&\leq
		\norm{c}_{q_\ell}+c_\ell\norm{u}_p^{p/q_\ell}
	\end{align*}
	and consequently $\norm{\sG^\ell(u)}_{q_\ell}\leq\norm{c}_{q_\ell}+c_\ell\norm{u}_p^{p/q_\ell}$. 

	(II) Claim: $\sG^\ell:L^p(\Omega,\K^d)\to L^{q_\ell}(\Omega,L_\ell(\K^d,\K^n))$ is continuous for all $0\leq\ell\leq m$.\\
	Let $(u_j)_{j\in\N}$ be a sequence in $L^p(\Omega,\K^d)$ converging to $u$ w.r.t.\ $\norm{\cdot}_p$. Because of \cite[p.~57, 3.22(1)]{alt:16} this implies pointwise convergence of a subsequence $\mu$-a.e., that is, w.l.o.g.\ we can suppose $\lim_{j\to\infty}u_j(x)=u_j(x)$ for $\mu$-a.a.\ $x\in\Omega$, which shows
	$$
		\sG^\ell(u_j)(x)=D_2^\ell g(x,u_j(x))
		\xrightarrow[j\to\infty]{}
		D_2^\ell g(x,u(x))
		=
		\sG^\ell(u)(x)
		\quad\text{for $\mu$-a.a.\ }x\in\Omega.
	$$
	The necessity part of Vitali's convergence theorem \cite[p.~57, 3.23]{alt:16} applies to $(u_j)_{j\in\N}$ and the growth assumptions on $D_2^\ell g$ guarantee that \cite[p.~57, 3.23]{alt:16} can be applied to $\sG^\ell(u_j)$ with the exponent $q_\ell$ rather than $p$. Thus, the sufficiency part of \cite[p.~57, 3.23]{alt:16} shows that $\lim_{j\to\infty}\norm{\sG^\ell(u_j)-\sG^\ell(u)}_{q_\ell}=0$, i.e.\ $\sG^\ell$ is continuous. 

	(III) Claim: $\bar\sG^{l+1}:L^p(\Omega,\K^d)\to L\bigl(L^p(\Omega,\K^d),L^{q_\ell}(\Omega,L_\ell(\K^d,\K^n))\bigr)$ is well-defined and continuous for all $0\leq\ell<m$.\\
	For $h\in L^p(\Omega,\K^d)$ we obtain from H\"older's inequality that
	\begin{align*}
		&
		\intoo{\int_\Omega\abs{D_2^{\ell+1}g(y,u(y))h(y)}^{q_\ell}\d\mu(y)}^{1/q_\ell}
		\leq
		\intoo{\int_\Omega\abs{D_2^{\ell+1}g(y,u(y))}^{q_\ell}\abs{h(y)}^{q_\ell}\d\mu(y)}^{1/q_\ell}\\
		\leq&
		\intoo{\intoo{\int_\Omega\abs{D_2^{\ell+1}g(y,u(y))}^{q_{\ell+1}}\d\mu(y)}^{q_\ell/q_{\ell+1}}\norm{h}_p^{q_\ell}}^{1/q_\ell}
		=
		\norm{\sG^{\ell+1}(u)}_{q_{\ell+1}}\norm{h}_p
	\end{align*}
	and consequently $\bar\sG^{\ell+1}(u)h\in L^{q_\ell}(\Omega,L_\ell(\K^d,\K^n))$ results. The assumptions on $g$ yield the inclusion $\bar\sG^{(\ell+1)}(u)\in L\bigl(L^p(\Omega,\K^d),L^{q_\ell}(\Omega,L_\ell(\K^d,\K^n))\bigr)$, while the continuity of $\bar\sG^\ell$ is shown as in step (II). 

	(IV) Claim: $\sG^\ell:L^p(\Omega,\K^d)\to L^{q_\ell}(\Omega,L_\ell(\K^d,\K^n))$ is continuously differentiable for all $0\leq\ell<m$.\\
	Because $D_2^\ell g$ is assumed to be continuously differentiable in the second argument, the Mean Value Theorem and Jensen's Inequality imply
	\begin{align*}
		&
		\norm{\sG^\ell(u+h)-\sG^\ell(u)-\bar\sG^{\ell+1}(u)h}_{q_\ell}\\
		\leq &
		\intoo{\int_\Omega\abs{D_2^\ell g(y,u(y)+h(y))-D_2^\ell g(y,u(y))-D_2^{\ell+1}g(x,u(y))h(y)}^{q_\ell}\d\mu(y)}^{1/{q_\ell}}\\
		\leq &
		\intoo{\int_\Omega\intoo{\int_0^1\abs{D_2^{\ell+1}g(y,u(y)+\vartheta h(y))-D_2^{\ell+1}g(x,u(y))]h(y)}\d\vartheta}^{q_\ell}\d\mu(y)}^{1/{q_\ell}}\\
		\leq &
		\intoo{\int_\Omega\int_0^1\abs{D_2^{\ell+1}g(y,u(y)+\vartheta h(y))-D_2^{\ell+1}g(x,u(y))]h(y)}^{q_\ell}\d\vartheta\d\mu(y)}^{1/{q_\ell}}
	\end{align*}
	and we obtain from Fubini's theorem along with H\"older's inequality
	\begin{align*}
		&
		\norm{\sG^\ell(u+h)-\sG^\ell(u)-\bar\sG^{\ell+1}(u)h}_{q_\ell}\\
		\leq &
		\intoo{\int_0^1\int_\Omega\abs{D_2^{\ell+1}g(y,u(y)+\vartheta h(y))-D_2^{\ell+1}g(x,u(y))]h(y)}^{q_\ell}\d\mu(y)\d\vartheta}^{1/{q_\ell}}\\
		\leq &
		\intoo{\int_0^1\int_\Omega\abs{D_2^{\ell+1}g(y,u(y)+\vartheta h(y))-D_2^{\ell+1}g(x,u(y))}^{q_\ell}\abs{h(y)}^{q_\ell}\d\mu(y)\d\vartheta}^{1/{q_\ell}}\\
		\leq &
		\intoo{\int_0^1\intoo{\int_\Omega\abs{D_2^{\ell+1}g(y,u(y)+\vartheta h(y))-D_2^{\ell+1}g(x,u(y))}^{q_{\ell+1}}\d\mu(y)}^{q_\ell /{q_{\ell+1}}}
		\norm{h}_p^{q_\ell}\d\vartheta}^{1/{q_\ell}}\\
		\leq &
		\intoo{\int_0^1\norm{\sG^{\ell+1}(u+\vartheta h)-\sG^{\ell+1}(u)}_{L^{q_{\ell+1}}(\Omega,L_\ell(\K^d,\K^n)))}^{q_\ell}\d\vartheta}^{1/q_\ell}\norm{h}_p
	\end{align*}
	for $h\in L^p(\Omega,\K^d)$. 
	Due to the continuity of $\sG^{\ell+1}:L^p(\Omega,\K^d)\to L^{q_{\ell+1}}(\Omega,L_{\ell+1}(\K^d,\K^n))$ the integral tends to zero in the limit $\norm{h}_p\to 0$, and hence $\sG^\ell$ is Frech{\'e}t differentiable in $u$ with derivative $D\sG^\ell(u)=\bar\sG^{\ell+1}(u)$ by uniqueness of derivatives. Its continuity was already shown in step (III). 

	(V) In particular, this establishes that $\sG:L^p(\Omega,\K^d)\to L^q(\Omega,\K^n)$ is continuously differentiable with $D\sG(u)=\bar\sG^1(u)$ as derivative. Then mathematical induction yields that $\sG$ is $m$-times continuously differentiable with the derivatives
	$$
		[D^\ell\sG(u)v_1\cdots v_\ell](x)=D_2^\ell g(x,u(x))v_1(x)\cdots v_\ell(x)
		\quad\text{for $\mu$-a.a.\ }x\in\Omega
	$$
	and all $v_1,\ldots,v_\ell\in L^p(\Omega,\K^d)$, $1\leq\ell\leq m$, which establishes claim (a). Moreover, the estimates stated in (b) are a consequence of step (I). 
\end{proof}
%
%

%
\end{document}